\documentclass[review]{elsarticle}

\usepackage{lineno,hyperref}
\usepackage{amsfonts, amsmath, amsthm, amssymb}
\usepackage[english]{babel}
\usepackage[mathcal]{eucal}
\usepackage{geometry}
\usepackage{relsize}
\usepackage{longtable}
\modulolinenumbers[5]

\newtheorem{pro}{Proposition}
\newtheorem{lem}[pro]{Lemma}
\newtheorem{thm}[pro]{Theorem}

\journal{Journal of Computational Statistics and Data Analysis}

\begin{document}

\begin{frontmatter}

\title{On the sample mean after a group sequential trial\footnote{A supplementary file with data accompanies the paper.}}

\author{Ben Berckmoes\footnote{
postal address: Universiteit Antwerpen, Departement Wiskunde-Informatica, Middelheimlaan 1, 2020 Antwerpen, Belgium; email: ben.berckmoes@uantwerpen.be.}, Anna Ivanova, Geert Molenberghs}

\begin{abstract}
A popular setting in medical statistics is a group sequential trial with independent and identically distributed normal outcomes, in which interim analyses of the sum of the outcomes are performed. Based on a prescribed stopping rule, one decides after each interim analysis whether the trial is stopped or continued. Consequently, the actual length of the study is a random variable. It is reported in the literature that the interim analyses may cause bias if one uses the ordinary sample mean to estimate the location parameter. For a generic stopping rule, which contains many classical stopping rules as a special case, explicit formulas for the expected length of the trial, the bias, and the mean squared error (MSE) are provided. It is deduced that, for a fixed number of interim analyses, the bias and the MSE converge to zero if the first interim analysis is performed not too early. In addition, optimal rates for this convergence are provided.  Furthermore, under a regularity condition, asymptotic normality in total variation distance for the sample mean is established. A conclusion for naive confidence intervals based on the sample mean is derived. It is also shown how the developed theory naturally fits in the broader framework of likelihood theory in a group sequential trial setting. A simulation study underpins the theoretical findings.
\end{abstract}


\end{frontmatter}

\linenumbers

\section{Introduction}\label{sec:Intro}

Throughout the paper, $X_1, X_2, \ldots$ will be a fixed sequence of independent and identically distributed random variables with law $N(\mu,\sigma^2)$, and $\psi_1, \psi_2, \ldots$ a fixed sequence of Borel measurable maps of $\mathbb{R}$ into $[0,1]$. 

For natural numbers $L \in \mathbb{N}_0$ and $0 < m_1 < m_2 < \ldots < m_L < n$, we consider a random sample size $N$ with the following properties:
\begin{itemize}
\item[(a)] $\displaystyle{N \text{ can take the values } m_1, m_2, \ldots, m_L, n,}\label{eq:ValuesN}$

\item[(b)] $\displaystyle{\forall i \in \{1,\ldots,L\} : \{N = m_i\} \text{ is independent of } X_{m_{i}+1}, X_{m_{i} + 2},\ldots,}$

\item[(c)] $\displaystyle{\forall i \in \{1,\ldots,L\}: \mathbb{P}\left[N = m_i \mid X_1, \ldots, X_{m_i}\right] = \psi_{m_i}(K_{m_i})\prod_{j=1}^{i - 1} \left[1 - \psi_{m_j}(K_{m_j})\right],}$
where we denote $K_m = \sum_{i=1}^m X_i$ and the empty product is 1.
\end{itemize}

The above setting serves as a paradigm for a group sequential trial of random length $N$ with outcomes $X_1,X_2,\ldots$ At each $m_i$, an interim analysis of the sum $K_{m_i}$ of the outcomes is performed and, based on the generic stopping rule (c), one decides whether the trial is stopped, i.e. $N = m_i$, or continued, i.e. $N > m_i$.

Note that the product in (c) is merely the usual decomposition of the conditional probability to reach a certain sample size and the product of conditional probabilities of continuing at smaller sample sizes, given that the trial is ongoing. This is similar to decompositions encountered in longitudinal or time-series transition models, and dropout models in longitudinal studies. It follows the law of total probability.

More precisely, at the $i$-th interim analysis only the values of the full sums $K_{m_1}, \ldots, K_{m_i}$ have been analyzed. Therefore,
\begin{eqnarray*}
\lefteqn{\mathbb{P}\left[N = m_i \mid X_1, \ldots, X_{m_i}\right]} \\
&=& \mathbb{P}\left[N = m_i \mid K_{m_1}, \ldots, K_{m_i}\right] \\
&=& \mathbb{P}\left[N = m_i, N \neq m_{i - 1}, \ldots, N \neq m_1 \mid K_{m_1}, \ldots, K_{m_i}\right],
\end{eqnarray*}
which, by the law of total probability,
\begin{eqnarray*}
&=& \mathbb{P}\left[N = m_i \mid N \neq m_{i -1}, \ldots, N \neq m_1, K_{m_1}, \ldots, K_{m_i}\right]\\
&& \prod_{j=1}^{i - 1} \mathbb{P}\left[N \neq m_j \mid N \neq m_{j - 1}, \ldots, N \neq m_1, K_{m_1}, \ldots, K_{m_i}\right],
\end{eqnarray*}
which, because, given that $N \neq m_{j - 1}, \ldots, N \neq m_1$, the event $\{N = m_j\}$ only depends on the analysis of the full sum $K_{m_j}$,
$$=  \mathbb{P}\left[N = m_i \mid N \neq m_{i -1}, \ldots, N \neq m_1 , K_{m_i}\right]  \prod_{j=1}^{i - 1} \mathbb{P}\left[N \neq m_j \mid N \neq m_{j - 1}, \ldots, N \neq m_1, K_{m_j}\right],$$
which is exactly the decomposition in (c).

We wish to highlight that the above model contains very useful settings that are extensively studied in the literature. To illustrate this, we let, for each $m$, 
\begin{equation*}
\psi_m(x) = 1_{\{\left|\cdot\right| \geq C_m\}}(x) = \left\{\begin{array}{clrr}      
1 &\textrm{ if }& \left|x\right| \geq C_m\\       
0 &\textrm{ if }&  \left|x\right| < C_m
\end{array}\right.,\label{eq:SimpeleBenadering}
\end{equation*}
with $C_m \in \mathbb{R}^+_0$ a constant.
For these choices of $\psi_m$, expression (c) is turned into
\begin{eqnarray*}
\lefteqn{\mathbb{P}[N = m_i \mid X_1, \ldots, X_{m_i}]}\label{eq:StoppingRuleSpecific}\\
&=& 1_{\{\left|\cdot\right| \geq C_{m_i}\}}(K_{m_i})\prod_{j=1}^{i - 1} \left[1 - 1_{\{\left|\cdot\right| \geq C_{m_j}\}}(K_{m_j})\right]\\
 &=& \left\{\begin{array}{clrr}     
1 &\textrm{ if }& \left|K_{m_i} \right| \geq C_{m_i} \text{ and } \forall j \in \{1,\ldots,i-1\} : \left|K_{m_j} \right| <  C_{m_j}\\       
0 &\textrm{ otherwise }&  \nonumber 
\end{array}\right..\nonumber
\end{eqnarray*}
So this corresponds to a trial which is stopped either at the first $m_i$ for which $\left|K_{m_i}\right| \geq C_{m_i}$, or at $n$. If, for a fixed constant $C \in \mathbb{R}^+_0$, $C_m = \sigma C \sqrt{m}$, this setting corresponds to the Pocock boundaries, studied in e.g. \cite{S78} and \cite{C89}, and, if, for a fixed constant $C \in \mathbb{R}^+_0$, $C_m = C$, this setting corresponds to the O'Brien-Fleming boundaries, studied in e.g. \cite{W92}. More generally, taking $\psi_m = 1_{\mathcal{S}_m}$ with $\mathcal{S}_m \subset \mathbb{R}$ a Borel measurable set, leads to the setting studied in e.g. \cite{EF90} and \cite{LH99}. Finally, taking $\psi_m(x) = \Phi\left(\alpha + \beta m^{-1} x\right)$ with $\Phi$ the standard normal cumulative distribution function and $\alpha,\beta$ real numbers, corresponds to the probabilistic stopping rule setting studied in e.g. \cite{MKA14}.

In this paper, we will study the ordinary sample mean $\widehat{\mu}_N = \frac{1}{N} K_N$. It is reported in the literature that in the above described group sequential trial setting, bias may occur if $\widehat{\mu}_N$ is used to estimate $\mu$ (\cite{HP88,EF90,LH99}). However, it was shown recently in \cite{MKA14} that if $N$ only takes the values $m$ and $2m$, and $\psi_m(x)$ takes the form $\Phi\left(\alpha + \beta m^{-1} x\right)$ or $\lim_{\beta \rightarrow \infty} \Phi\left(\alpha + \beta m^{-1} x\right) = 1_{\{\cdot \geq 0\}}(x)$, this bias vanishes as $m$ tends to $\infty$. In this paper, we will establish explicit formulas for the expected length of the trial, the bias, and the mean squared error (MSE) in the general case, described by (a), (b), and (c). We deduce that, for fixed $L$, if $m_1 \to \infty$ (and hence $\forall i : m_i \rightarrow \infty$ and $n \rightarrow \infty$), the bias vanishes with rate $1/\sqrt{m_1}$, and the MSE vanishes with rate $1/m_1$. We will show that both rates are optimal. Furthermore,  under a regularity condition, we will establish asymptotic normality in total variation distance for the sample mean if, for fixed $L$ and $m_1,\ldots,m_L$, $n \to \infty$. In some cases, this validates the use of naive confidence intervals based on the sample mean if $n$ is large.

The paper is structured as follows. In section \ref{sec:NormalTransform}, we introduce the normal transform of a finite tuple of bounded Borel measurable maps of $\mathbb{R}$ into $\mathbb{R}$, for which we establish a recursive formula. We use the normal transform in section \ref{sec:JointDensity} to obtain an explicit formula for the joint density of $N$ and $K_N$. We establish a fundamental result in section \ref{sec:FunRes}, which is used to calculate the expected length of the trial in section \ref{sec:EL}, and the bias and the MSE in section \ref{sec:Bias}. It is shown that, for fixed $L$, if $m_1 \to \infty$ (and hence $\forall i : m_i \rightarrow \infty$ and $n \rightarrow \infty$), the bias vanishes with rate $1/\sqrt{m_1}$, and the MSE vanishes with rate $1/m_1$. Both rates are shown to be optimal. In section \ref{sec:AsymNor}, under a regularity condition, we establish asymptotic normality in total variation distance for $\widehat{\mu}_N$ if, for fixed $L$ and $m_1,\ldots,m_L$, $n \rightarrow \infty$. We also derive a conclusion for naive confidence intervals based on $\widehat{\mu}_N$. In section \ref{sec:LHT}, we show how the theory developed in this paper fits in the broader framework of likelihood theory. A simulation study, which underpins our theoretical results, is conducted in section \ref{sec:Simulations}. Finally, some concluding remarks are formulated in section \ref{sec:ConRem}.

\section{The normal transform}\label{sec:NormalTransform}

Let $\phi$ be the standard normal density. For a finite tuple $B = (b_1,\ldots,b_i)$ of bounded Borel measurable maps of $\mathbb{R}$ into $\mathbb{R}$, we define the {\em normal transform} to be the map $\mathcal{N}_{B,\mu,\sigma}$ of $]0,\infty[^{i + 1} \times \mathbb{R}$ into $\mathbb{R}$ given by
\begin{eqnarray}
\lefteqn{\mathcal{N}_{B,\mu,\sigma}(x_1,\ldots,x_{i+1},x)}\label{eq:NormalTransform}\\ 
&=& \frac{\int_{-\infty}^\infty \ldots \int_{-\infty}^\infty \prod_{j=1}^i \frac{\phi\left(\frac{z_j - \mu x_j}{\sigma \sqrt{x_j}}\right)}{\sigma \sqrt{x_j}} b_j\left(\sum_{k=1}^j z_k\right) \frac{\phi\left(\frac{x - \sum_{k=1}^i z_k - \mu x_{i + 1}}{\sigma \sqrt{x_{i + 1}}}\right)}{\sigma \sqrt{x_{i + 1}}} dz_i \ldots dz_1}{\frac{1}{\sigma \sqrt{\sum_{k=1}^{i + 1} x_k}}\phi\left(\frac{x - \mu\sum_{k=1}^{i + 1} x_k}{\sigma \sqrt{\sum_{k=1}^{i+1} x_k}}\right)}.\nonumber
 \end{eqnarray}
 
 We will provide a recursive formula for the normal transform in Theorem \ref{thm:RecursiveDescriptionNormalTransform}. We need two lemmas.
 
\begin{lem}
For $x_1,x_2 \in ]0,\infty[$ and $x,z \in \mathbb{R}$,
\begin{equation}
\phi\left(\frac{z - \mu x_1}{\sigma\sqrt{x_1}}\right)  \phi\left(\frac{x - z - \mu x_2}{\sigma\sqrt{x_2}}\right) = \phi\left(\frac{x - \mu (x_1+x_2)}{\sigma\sqrt{x_1 + x_2}}\right) \phi\left(\frac{\frac{x_1 + x_2}{x_1} z - x}{\sigma \sqrt{ \frac{x_2(x_1 + x_2)}{x_1}}}\right).\label{eq:TwoGaussianPdf}
\end{equation}
\end{lem}

\begin{proof}
This is readily verified by a straightforward calculation.
\end{proof}

\begin{lem}\label{lem:IntTwoGaussians}
Let $\xi$ be a random variable with law $N(0,1)$. For a bounded Borel measurable map $b$ of $\mathbb{R}$ into $\mathbb{R}$, $x_1,x_2 \in \left]0,\infty\right[$, and $x \in \mathbb{R}$,
\begin{eqnarray}
\lefteqn{\int_{-\infty}^\infty \frac{1}{\sigma \sqrt{x_1}}\phi\left(\frac{z - \mu x_1}{\sigma \sqrt{x_1}}\right) b(z) \frac{1}{\sigma \sqrt{x_2}} \phi\left(\frac{x - z - \mu x_2}{\sigma \sqrt{x_2}}\right) dz}\label{eq:Integral2}\\
&=& \frac{1}{\sigma \sqrt{x_1 + x_2}} \phi\left(\frac{x - \mu (x_1 + x_2)}{\sigma \sqrt{x_1 + x_2}}\right)  \mathbb{E}\left[b\left(\frac{x_1}{x_1 + x_2} x + \sigma \sqrt{\frac{x_1 x_2}{x_1 + x_2}} \xi\right)\right].\nonumber
\end{eqnarray}
\end{lem}

\begin{proof}
By (\ref{eq:TwoGaussianPdf}),
\begin{eqnarray*}
\lefteqn{\int_{-\infty}^\infty \frac{1}{\sigma \sqrt{x_1}}\phi\left(\frac{z - \mu x_1}{\sigma \sqrt{x_1}}\right) b(z) \frac{1}{\sigma \sqrt{x_2}} \phi\left(\frac{x - z - \mu x_2}{\sigma \sqrt{x_2}}\right) dz}\\
&=& \frac{1}{\sigma^2 \sqrt{x_1 x_2}} \phi\left(\frac{x - \mu(x_1 + x_2)}{\sigma \sqrt{x_1 + x_2}}\right) \int_{-\infty}^\infty b(z) \phi\left(\frac{\frac{x_1 + x_2}{x_1} z - x}{\sigma \sqrt{\frac{x_2 (x_1 + x_2)}{x_1}}}\right) dz,
\end{eqnarray*}
which is seen to coincide with the right-hand side of (\ref{eq:Integral2}) after performing the change of variables $u = \frac{\frac{x_1 + x_2}{x_1} z - x}{\sigma \sqrt{\frac{x_2 (x_1 + x_2)}{x_1}}}$. This finishes the proof.
\end{proof}

\begin{thm}\label{thm:RecursiveDescriptionNormalTransform}
Let $\xi$ be a random variable with law $N(0,1)$. For a bounded Borel measurable map $b$ of $\mathbb{R}$ into $\mathbb{R}$, $x_1,x_2 \in \left]0,\infty\right[$, and $x \in \mathbb{R}$,
\begin{equation}
\mathcal{N}_{b,\mu,\sigma}(x_1,x_2,x) =  \mathbb{E}\left[b\left(\frac{x_1}{x_1 + x_2} x + \sigma \sqrt{\frac{x_1 x_2}{x_1 + x_2}} \xi\right)\right].\label{eq:NormalTransform1}
\end{equation}
Furthermore, for a natural number $i \geq 2$, a tuple $(b_1,\ldots,b_i)$ of bounded Borel measurable maps of $\mathbb{R}$ into $\mathbb{R}$, $x_1,\ldots,x_{i+1} \in \left]0,\infty\right[$, and $x \in \mathbb{R}$,
\begin{equation}
\mathcal{N}_{(b_1,\ldots,b_{i}), \mu,\sigma}(x_1,\ldots,x_{i+1},x) = \mathcal{N}_{(b_1,\ldots,b_{i-2},\widetilde{b}_{i-1}),\mu,\sigma}(x_1,\ldots,x_{i-1},x_{i} + x_{i + 1},x),\label{eq:NormRec}
\end{equation}
where
\begin{equation}
\widetilde{b}_{i-1}(z) = b_{i-1}(z) \mathbb{E}\left[b_i\left(\frac{x_{i+1}}{x_i + x_{i+1}} z + \frac{x_i}{x_i + x_{i + 1}} x + \sigma \sqrt{\frac{x_i x_{i+1}}{x_i + x_{i+1}}} \xi\right)\right].\label{eq:bTilde}
\end{equation}
\end{thm}

\begin{proof}
Formula (\ref{eq:NormalTransform1}) follows directly from Lemma \ref{lem:IntTwoGaussians}.

We now establish formula (\ref{eq:NormRec}). We have
\begin{eqnarray*}
&&\int_{-\infty}^\infty \ldots \int_{-\infty}^\infty \prod_{j=1}^i \frac{\phi\left(\frac{z_j - \mu x_j}{\sigma \sqrt{x_j}}\right)}{\sigma \sqrt{x_j}} b_j\left(\sum_{k=1}^j z_k\right) \frac{\phi\left(\frac{x - \sum_{k=1}^i z_k - \mu x_{i + 1}}{\sigma \sqrt{x_{i + 1}}}\right)}{\sigma \sqrt{x_{i + 1}}} dz_i \ldots dz_1\\
&&= \int_{-\infty}^\infty \ldots \int_{-\infty}^\infty \prod_{j=1}^{i-1} \frac{\phi\left(\frac{z_j - \mu x_j}{\sigma \sqrt{x_j}}\right)}{\sigma \sqrt{x_j}} b_j\left(\sum_{k=1}^j z_k\right)\\
&&\phantom{abit} \left(\int_{-\infty}^\infty  \frac{\phi\left(\frac{z_i - \mu x_i}{\sigma \sqrt{x_i}}\right)}{\sigma \sqrt{x_i}} b_i \left(\sum_{k=1}^i z_k\right) \frac{\phi\left(\frac{x - \sum_{k=1}^i z_k - \mu x_{i + 1}}{\sigma \sqrt{x_{i + 1}}}\right)}{\sigma \sqrt{x_{i + 1}}} dz_i\right)dz_{i-1} \ldots dz_i,
\end{eqnarray*}
which, applying (\ref{eq:Integral2}) to the map $b(z) = b_i\left(\sum_{k=1}^{i-1} z_k +  z\right)$ in the integration with respect to $z_i$, reduces to
\begin{eqnarray*}
&&\int_{-\infty}^\infty \ldots \int_{-\infty}^\infty \prod_{j=1}^{i-1} \frac{\phi\left(\frac{z_j - \mu x_j}{\sigma \sqrt{x_j}}\right)}{\sigma \sqrt{x_j}} b_j\left(\sum_{k=1}^j z_k\right) \\
&&\phantom{abit}\mathbb{E}\left[\sum_{k=1}^{i-1} z_k + \frac{x_i}{x_i + x_{i+1}}\left(x - \sum_{k=1}^{i-1} z_k\right) + \sigma \sqrt{\frac{x_i x_{i+1}}{x_i + x_{i+1}}} \xi\right] \\
&&\phantom{abit} \frac{1}{\sigma \sqrt{x_i + x_{i+1}}} \phi\left(\frac{x - \sum_{k=1}^{i-1} z_k - \mu(x_i + x_{i+1})}{\sigma \sqrt{x_i + x_{i+1}}}\right) dz_{i-1} \ldots dz_1,
\end{eqnarray*}
which, using notation (\ref{eq:bTilde}), equals
\begin{eqnarray*}
&&\int_{-\infty}^\infty \ldots \int_{-\infty}^\infty \prod_{j=1}^{i-2} \frac{\phi\left(\frac{z_j - \mu x_j}{\sigma \sqrt{x_j}}\right)}{\sigma \sqrt{x_j}} b_j\left(\sum_{k=1}^j z_k\right) \frac{\phi\left(\frac{z_{i-1} - \mu x_{i-1}}{\sigma \sqrt{x_{i-1}}}\right)}{\sigma \sqrt{x_{i-1}}} \widetilde{b}_{i-1}\left(\sum_{k=1}^{i-1} z_k\right)\\
 &&\phantom{abit} \frac{1}{\sigma \sqrt{x_i + x_{i+1}}} \phi\left(\frac{x - \sum_{k=1}^{i-1} z_k - \mu(x_i + x_{i+1})}{\sigma \sqrt{x_i + x_{i+1}}}\right) dz_{i-1} \ldots dz_1,
\end{eqnarray*}
which, by definition (\ref{eq:NormalTransform}),
\begin{equation*}
= \frac{1}{\sigma \sqrt{\sum_{k=1}^{i + 1} x_k}}\phi\left(\frac{x - \mu\sum_{k=1}^{i + 1} x_k}{\sigma \sqrt{\sum_{k=1}^{i+1} x_k}}\right) \mathcal{N}_{(b_1,\ldots,b_{i-2},\widetilde{b}_{i-1}),\mu,\sigma}(x_1,\ldots,x_{i-1},x_i + x_{i+1},x).
\end{equation*}
This finishes the proof.
\end{proof}

\section{The joint density of $N$ and $K_N$}\label{sec:JointDensity}

We return to the setting of the first section. Let $f_{N,K_N}(m,x)$ be the joint density of $N$ and $K_N$. Furthermore, put 
$$\Delta m_1 = m_1$$
and
\begin{equation}
\mathcal{N}_1(x) = 1,\label{eq:N1}
\end{equation}
and, for $i \in \{2,\ldots,L\}$, 
$$\Delta_{m_i} = m_i - m_{i-1}$$ 
and
\begin{equation}
\mathcal{N}_i(x) = \mathcal{N}_{\left(1 - \psi_{m_1}, \ldots, 1 - \psi_{m_{i-1}}\right),\mu, \sigma}(\Delta_{m_1},\ldots,\Delta_{m_i},x).\label{eq:Ni}
\end{equation}

We first establish in Theorem \ref{lem:NBetween0And1} that each $\mathcal{N}_i$ takes values between $0$ and $1$. We need the following lemma.

\begin{lem}\label{lem:NormalTransformBetween0And1}
Let $(b_1,\ldots,b_i)$ be a tuple of bounded Borel measurable maps of $\mathbb{R}$ into $\mathbb{R}$. If each $b_i$ takes values between $0$ and $1$, then $\mathcal{N}_{(b_1,\ldots,b_{i}), \mu,\sigma}$ takes values between $0$ and $1$.
\end{lem}

\begin{proof}
Using formulas (\ref{eq:NormalTransform1}), (\ref{eq:NormRec}), and (\ref{eq:bTilde}), this follows easily by an inductive argument.
\end{proof}

\begin{thm}\label{lem:NBetween0And1}
For each $i \in \{1,\ldots,L\}$, $\mathcal{N}_i$ takes values between $0$ and $1$.
\end{thm}

\begin{proof}
Using (\ref{eq:N1}), (\ref{eq:Ni}), and the fact that each $\psi_i$ takes values between $0$ and $1$, this follows from Lemma \ref{lem:NormalTransformBetween0And1}.
\end{proof}

The importance of the normal transform is reflected by the following result, which provides a formula for the joint density of $N$ and $K_N$ at the places where the interim analyses are performed. 

\begin{thm}\label{thm:JointDensity}
For $i \in \{1,\ldots,L\}$,
\begin{equation}
f_{N,K_N}(m_i,x) = \frac{1}{\sigma \sqrt{m_i}}\phi\left(\frac{x - \mu m_i}{\sigma \sqrt{m_i}}\right) \psi_{m_i} (x)\mathcal{N}_i(x).\label{eq:JointDensity}
\end{equation}
\end{thm}

\begin{proof}
We first consider the case $i  = 1$. We have
\begin{equation}
f_{N,K_{N}}(m_1,x) = f_{N,K_{m_1}}(m_1,x) = f_{N \mid K_{m_1}}(m_1\mid x) f_{K_{m_1}}(x),\label{eq:JointDensity11}
\end{equation}
with $f_{N \mid K_{m_1}}(m \mid x)$ the conditional density of $N$ given $K_{m_1}$ and $f_{K_{m_1}}$ the density of $K_{m_1}$. By condition (c) in section 1, and using the discrete nature of $N$, 
\begin{equation}
f_{N \mid K_{m_1}}(m_1\mid x) = \psi_{m_1}(x).\label{eq:JointDensity12}
\end{equation}
Furthermore, since the $X_k$ are independent with distribution $N(\mu,\sigma^2)$,
\begin{equation}
f_{K_{m_1}}(x) = \frac{1}{\sigma \sqrt{m_1}} \phi\left(\frac{x - \mu m_1}{\sigma \sqrt{m_1}}\right).\label{eq:JointDensity13}
\end{equation}
Combining (\ref{eq:JointDensity11}), (\ref{eq:JointDensity12}), and (\ref{eq:JointDensity13}), shows that (\ref{eq:JointDensity}) holds in the case $i = 1$.

We now turn to the case $i \geq 2$. Put 
$$S_{m_1} = K_{m_1}$$
and, for $j \in \{2,\ldots,L\}$, 
$$S_{m_j} = K_{m_{j}} - K_{m_{j-1}}.$$ 
Let 
$$f_{N,S_{m_1},\ldots,S_{m_i}}(m,x_1,\ldots,x_i)$$ 
be the joint density of $N$ and $S_{m_1},\ldots,S_{m_i}$, 
$$f_{N \mid S_{m_1}, \ldots, S_{m_i}}(m \mid x_1, \ldots, x_i)$$ 
the conditional density of $N$ given $S_{m_1}, \ldots, S_{m_i}$, 
and 
$$f_{S_{m_1},\ldots,S_{m_i}}(x_1,\ldots,x_i)$$ 
the joint density of $S_{m_1},\ldots,S_{m_i}$. Then
\begin{eqnarray}
\lefteqn{f_{N,K_N}(m_i,x)}\label{eq:JointDensity21}\\
&=& f_{N,K_{m_i}}(m_i,x)\nonumber\\
&=& \int_{-\infty}^\infty \cdots \int_{-\infty}^\infty f_{N,S_{m_1},\ldots,S_{m_i}}\left(m_i,z_1,\ldots,z_{i-1},x-\sum_{k=1}^{i-1} z_k\right) dz_{i-1} \ldots dz_1\nonumber\\
&=& \int_{-\infty}^\infty \cdots \int_{-\infty}^\infty f_{N \mid S_{m_1},\ldots,S_{m_i}}\left(m_i \mid z_1,\ldots,z_{i-1},x-\sum_{k=1}^{i-1} z_k\right)\nonumber\\
&&\phantom{abitfurther}f_{S_{m_1},\ldots,S_{m_i}}\left(z_1,\ldots,z_{i-1},x - \sum_{k=1}^{i-1} z_k\right) dz_{i-1} \ldots dz_1.\nonumber 
\end{eqnarray}
By condition (c) in section 1, and using the discrete nature of $N$, 
\begin{equation}
f_{N \mid S_{m_1},\ldots,S_{m_i}}\left(m_i \mid z_1,\ldots,z_{i-1},x-\sum_{k=1}^{i-1} z_k\right) = \psi_{m_i}(x)\prod_{j=1}^{i - 1} \left[1 - \psi_{m_j}\left(\sum_{k=1}^j z_k\right)\right].\label{eq:JointDensity22}
\end{equation}
Furthermore, the $X_k$ being independent with distribution $N(\mu,\sigma^2)$,
\begin{eqnarray}
\lefteqn{f_{S_{m_1},\ldots,S_{m_i}}\left(z_1,\ldots,z_{i-1},x - \sum_{k=1}^{i-1} z_k\right)}\label{eq:JointDensity23}\\
&=& \prod_{j=1}^{i-1} f_{S_{m_j}}(z_j) f_{S_{m_i}}\left(x - \sum_{k=1}^{i-1} z_k\right)\nonumber\\
&=& \prod_{j=1}^{i - 1} \frac{1}{\sigma \sqrt{\Delta m _j}} \phi\left(\frac{z_j - \mu \Delta m_j}{\sigma \sqrt{\Delta m_j}}\right) \frac{1}{\sigma \sqrt{\Delta m_i}} \phi\left(\frac{x - \sum_{k=1}^{i-1} z_k - \mu \Delta m_i}{\sigma \sqrt{\Delta m_i}}\right).\nonumber
\end{eqnarray}
Combining  definition (\ref{eq:NormalTransform}) with (\ref{eq:JointDensity21}), (\ref{eq:JointDensity22}), and (\ref{eq:JointDensity23}), establishes (\ref{eq:JointDensity}). This finishes the proof.
\end{proof}

Finally, we will provide a formula for the joint density of $N$ and $K_N$ at $n$ in Theorem \ref{thm:JointDensity2}. We need the following lemma.

\begin{lem}
Let $\xi$ be a random variable with law $N(0,1)$. Then, for $i \in \{1,\ldots,L\}$,
\begin{equation}
f_{N,K_n}(m_i,x) = \frac{1}{\sigma \sqrt{n}} \phi\left(\frac{x - \mu n}{\sigma \sqrt{n}}\right) \mathbb{E}\left[\left(\psi_{m_i} \mathcal{N}_i\right)\left(\frac{m_i}{n} x + \sigma \sqrt{\frac{m_i (n - m_i)}{n}} \xi\right)\right].\label{eq:auxJointDensity2}
\end{equation}
\end{lem}

\begin{proof}
We have 
$$f_{N,K_n}(m_i,x) = \int_{-\infty}^\infty f_{N,K_{m_i}, K_n - K_{m_i}}(m_i,z, x-z) dz,$$
which, by condition (b) in section 1,
$$= \int_{-\infty}^\infty f_{N,K_{m_i}} (m_i,z) f_{K_n - K_{m_i}} (x - z) dz,$$
which, by (\ref{eq:JointDensity}),
$$= \int_{-\infty}^\infty \frac{1}{\sigma \sqrt{m_i}}\phi\left(\frac{z - \mu m_i}{\sigma \sqrt{m_i}}\right) \psi_{m_i} (z)\mathcal{N}_i(z) \frac{1}{\sigma \sqrt{n-m_i}} \phi\left(\frac{x - z - \mu \sqrt{n - m_i}}{\sigma\sqrt{n - m_i}}\right),$$
which, by (\ref{eq:Integral2}),
$$= \frac{1}{\sigma \sqrt{n}} \phi\left(\frac{x - \mu n}{\sigma \sqrt{n}}\right) \mathbb{E}\left[\left(\psi_{m_i} \mathcal{N}_i\right)\left(\frac{m_i}{n} x + \sigma \sqrt{\frac{m_i (n - m_i)}{n}}  \xi\right)\right].$$
This finishes the proof.
\end{proof}

\begin{thm}\label{thm:JointDensity2}
Let $\xi$ be a random variable with law $N(0,1)$. Then
\begin{equation}
f_{N,K_N}(n,x) = \frac{1}{\sigma\sqrt{n}} \phi\left(\frac{x - \mu n}{\sigma \sqrt{n}}\right)\left[1 - \sum_{i=1}^L \mathbb{E}\left[\left(\psi_{m_i} \mathcal{N}_i\right)\left(\frac{m_i}{n} x + \sigma\sqrt{\frac{m_i (n - m_i)}{n}} \xi\right)\right]\right].\label{eq:JointDensity2}
\end{equation}
\end{thm}

\begin{proof}
By condition (a) in section 1,
$$f_{N,K_N}(n,x) = f_{N,K_n}(n,x) = f_{K_n}(x) - \sum_{i=1}^L f_{N,K_n}(m_i,x),$$
which, applying (\ref{eq:auxJointDensity2}), proves the desired result.
\end{proof}

\section{A fundamental result}\label{sec:FunRes}

We will prove Theorem \ref{thm:FunRes}, which will play a fundamental role in the calculation of the expected length of the trial, the bias, and the MSE, and in the establishment of an asymptotic normality result. We need the following lemma. 

\begin{lem}
Let $\eta$ be a random variable with law $N(0,1)$, $g$ a Borel measurable map of $\mathbb{R}$ into $\mathbb{R}$ with
$\mathbb{E}[\left|g(\eta)\right|] < \infty$, and $m \in \mathbb{R}^+_0$. Then
\begin{equation}
\int_{-\infty}^\infty \frac{1}{\sigma \sqrt{m}} \phi\left(\frac{x - \mu m}{\sigma \sqrt{m}}\right) g(x) dx = \mathbb{E}\left[g (\mu m + \sigma \sqrt{m} \eta)\right].\label{eq:NActsOnG0}
\end{equation}
\end{lem}

\begin{proof}
Perform the change of variables $z = \frac{x - \mu m}{\sigma \sqrt{m}}$.
\end{proof}

\begin{thm}\label{thm:FunRes}
Let $\xi$ and $\eta$ be independent random variables with law $N(0,1)$ and $h$ a Borel measurable map of $\mathbb{R}$ into $\mathbb{R}$ with $\mathbb{E}[\left|h(\eta)\right|] < \infty$. Then, for each $i \in \left\{1,\ldots,L\right\}$,
\begin{equation}
\mathbb{E}\left[h(\widehat{\mu}_N) 1_{\{N = m_i\}}\right] = \mathbb{E}\left[h\left(\mu + \frac{\sigma}{\sqrt{m_i}} \xi\right) \left(\psi_{m_i} \mathcal{N}_i\right)\left(\mu m_i + \sigma \sqrt{m_i} \xi\right)\right],\label{eq:Fun1}
\end{equation}
and 
\begin{eqnarray}
\lefteqn{\mathbb{E}\left[h(\widehat{\mu}_N) 1_{\{N = n\}}\right]}\label{eq:Fun2}\\
&=& \mathbb{E}\left[h\left(\mu + \frac{\sigma}{\sqrt{n}} \eta\right) \left(1 - \sum_{i=1}^L \mathbb{E}\left[\left(\psi_{m_i} \mathcal{N}_i\right)\left(\mu m_i + \sigma \sqrt{\frac{m_i(n - m_i)}{n}} \xi + \sigma \frac{m_i}{\sqrt{n}} \eta\right)\right]\right) \right].\nonumber
\end{eqnarray}
\end{thm}

\begin{proof}
By (\ref{eq:JointDensity}), for $i \in \{1,\ldots,L\}$,
$$\mathbb{E}\left[h(\widehat{\mu}_N) 1_{\{N = m_i\}}\right] = \int_{-\infty}^\infty h\left(\frac{1}{m_i} x\right) \frac{1}{\sigma \sqrt{m_i}}\phi\left(\frac{x - \mu m_i}{\sigma \sqrt{m_i}}\right) \psi_{m_i} (x)\mathcal{N}_i(x) dx,$$
which, using (\ref{eq:NActsOnG0}) with $g(x) = h\left(\frac{1}{m_i} x\right) \psi_{m_i} (x)\mathcal{N}_i(x)$ and  $m = m_i$, gives (\ref{eq:Fun1}).

Furthermore, by (\ref{eq:JointDensity2}),
\begin{eqnarray*}
\lefteqn{\mathbb{E}\left[h(\widehat{\mu}_N) 1_{\{N = n\}}\right]}\\
&=& \int_{-\infty}^\infty h\left(\frac{1}{n} x\right) \frac{1}{\sigma\sqrt{n}} \phi\left(\frac{x - \mu n}{\sigma \sqrt{n}}\right)\left[1 - \sum_{i=1}^L \mathbb{E}\left[\left(\psi_{m_i} \mathcal{N}_i\right)\left(\frac{m_i}{n} x + \sigma\sqrt{\frac{m_i (n - m_i)}{n}} \xi\right)\right]\right],
\end{eqnarray*}
which, applying (\ref{eq:NActsOnG0}) with $g(x) = h\left(\frac{1}{n} x\right)\left[1 - \sum_{i=1}^L \mathbb{E}\left[\left(\psi_{m_i} \mathcal{N}_i\right)\left(\frac{m_i}{n} x + \sigma\sqrt{\frac{m_i (n - m_i)}{n}} \xi\right)\right]\right]$ and $m = n$, and using independence of $\xi$ and $\eta$, gives (\ref{eq:Fun2}).
\end{proof}

\section{The expected length of the trial}\label{sec:EL}

The following result provides explicit formulas for the marginal density of the actual length of the trial $N$.

\begin{thm}\label{thm:MarginalProbabilities}
Let $\xi$ and $\eta$ be independent random variables with law $N(0,1)$. Then, for each $i \in \left\{1,\ldots,L\right\}$,
\begin{equation}
\mathbb{P}[N = m_i] = \mathbb{E}\left[\left(\psi_{m_i} \mathcal{N}_i\right)\left(\mu m_i + \sigma \sqrt{m_i} \xi\right)\right],\label{eq:PNIsmi}
\end{equation}
and
\begin{equation}
\mathbb{P}[N = n] = 1 - \sum_{i=1}^L \mathbb{E}\left[\left(\psi_{m_i} \mathcal{N}_i\right)\left(\mu m_i + \sigma \sqrt{\frac{m_i(n - m_i)}{n}} \xi + \sigma \frac{m_i}{\sqrt{n}} \eta\right)\right].\label{eq:PNIsn}
\end{equation}
\end{thm}

\begin{proof}
Applying (\ref{eq:Fun1}) with $h(x) = 1$, gives (\ref{eq:PNIsmi}). Furthermore, applying (\ref{eq:Fun2}) with $h(x) = 1$, gives (\ref{eq:PNIsn}).
\end{proof}

Next, we provide an explicit formula for the expected length of the trial.

\begin{thm}
Let $\xi$ and $\eta$ be independent random variables with law $N(0,1)$. Then 
\begin{eqnarray}
\mathbb{E}[N] &=& \sum_{i=1}^L m_i \mathbb{E}\left[\left(\psi_{m_i} \mathcal{N}_i\right)\left(\mu m_i + \sigma \sqrt{m_i} \xi\right)\right]\label{eq:ExL}\\
&&+ n \left(1 - \sum_{i=1}^L \mathbb{E}\left[\left(\psi_{m_i} \mathcal{N}_i\right)\left(\mu m_i + \sigma \sqrt{\frac{m_i(n - m_i)}{n}} \xi + \sigma \frac{m_i}{\sqrt{n}} \eta\right)\right]\right).\nonumber
\end{eqnarray}
\end{thm}

\begin{proof}
This follows immediately from Theorem \ref{thm:MarginalProbabilities}.
\end{proof}

\section{The bias and the mean squared error}\label{sec:Bias}

The following result provides an explicit formula for the bias if $\widehat{\mu}_N$ is used to estimate $\mu$.

\begin{thm}\label{thm:Bias}
Let $\xi$ and $\eta$ be independent random variables with law $N(0,1)$. Then
\begin{eqnarray}
\lefteqn{\mathbb{E}[\widehat{\mu}_N - \mu]}\label{eq:Bias}\\ 
&=&  \sum_{i=1}^L \frac{\sigma}{\sqrt{m_i}}\mathbb{E}\left[\xi\left(\psi_{m_i} \mathcal{N}_i\right)\left(\mu m_i + \sigma \sqrt{m_i} \xi\right)\right]\nonumber\\
&& -  \frac{\sigma}{\sqrt{n}} \sum_{i=1}^L\mathbb{E}\left[ \eta \left(\psi_{m_i} \mathcal{N}_i\right)\left(\mu m_i + \sigma \sqrt{\frac{m_i(n - m_i)}{n}} \xi + \sigma \frac{m_i}{\sqrt{n}} \eta\right)\right].\nonumber
\end{eqnarray}
\end{thm}

\begin{proof}
We have
$$\mathbb{E}[\widehat{\mu}_N - \mu] = \sum_{i = 1}^L \mathbb{E}[\left(\widehat{\mu}_N - \mu\right) 1_{\{N = m_i\}}] + \mathbb{E}\left[\left(\widehat{\mu}_N - \mu\right)1_{\{N=n\}}\right].$$
Now (\ref{eq:Bias}) follows by applying (\ref{eq:Fun1}) and (\ref{eq:Fun2}) with $h(x) = x - \mu$.
\end{proof}

From Theorem \ref{thm:Bias}, we derive the following universal bound for the bias.

\begin{thm}\label{thm:BiasBound}
\begin{equation}
\left|\mathbb{E}[\widehat{\mu}_N - \mu]\right| \leq \sigma \sqrt{\frac{2}{\pi}} \left(\sum_{i=1}^L \frac{1}{\sqrt{m_i}} + \frac{L}{\sqrt{n}}\right).\label{eq:BiasBound}
\end{equation}
In particular, the bias vanishes if, for fixed $L$, $m_1 \to \infty$ (and hence $\forall i : m_i \rightarrow \infty$ and $n \rightarrow \infty$).
\end{thm}

\begin{proof}
Theorem \ref{lem:NBetween0And1} shows that each $\psi_{m_i} \mathcal{N}_i$ takes values between $0$ and $1$. Therefore, for $\xi$ and $\eta$ with law $N(0,1)$,
$$\left|\mathbb{E}\left[\xi\left(\psi_{m_i} \mathcal{N}_i\right)\left(\mu m_i + \sigma \sqrt{m_i} \xi\right)\right]\right| \leq \mathbb{E}[\left|\xi\right|] = \sqrt{\frac{2}{\pi}}$$
and
$$\left|\mathbb{E}\left[ \eta \left(\psi_{m_i} \mathcal{N}_i\right)\left(\mu m_i + \sigma \sqrt{\frac{m_i(n - m_i)}{n}} \xi + \sigma \frac{m_i}{\sqrt{n}} \eta\right)\right]\right| \leq \mathbb{E}[\left|\eta\right|] = \sqrt{\frac{2}{\pi}}.$$
Now (\ref{eq:BiasBound}) follows easily from (\ref{eq:Bias}).
\end{proof}

The following result provides a formula, similar to (\ref{eq:Bias}), for the mean squared error (MSE).

\begin{thm}\label{thm:MSE}
Let $\xi$ and $\eta$ be independent random variables with law $N(0,1)$. Then
\begin{eqnarray}
\lefteqn{\mathbb{E}\left[\left(\widehat{\mu}_N - \mu\right)^2\right]}\label{eq:MSE}\\
&=&  \sum_{i=1}^L \frac{\sigma^2}{m_i}\mathbb{E}\left[\xi^2\left(\psi_{m_i} \mathcal{N}_i\right)\left(\mu m_i + \sigma \sqrt{m_i} \xi\right)\right]\nonumber\\
&& + \frac{\sigma^2}{n} - \frac{\sigma^2}{n} \sum_{i=1}^L\mathbb{E}\left[ \eta^2 \left(\psi_{m_i} \mathcal{N}_i\right)\left(\mu m_i + \sigma \sqrt{\frac{m_i(n - m_i)}{n}} \xi + \sigma \frac{m_i}{\sqrt{n}} \eta\right)\right].\nonumber
\end{eqnarray}
\end{thm}

\begin{proof}
We have
$$\mathbb{E}\left[\left(\widehat{\mu}_N - \mu\right)^2\right] = \sum_{i = 1}^L \mathbb{E}\left[\left(\widehat{\mu}_N - \mu\right)^2 1_{\{N = m_i\}}\right] + \mathbb{E}\left[\left(\widehat{\mu}_N - \mu\right)^21_{\{N=n\}}\right].$$
Now (\ref{eq:MSE}) follows by applying (\ref{eq:Fun1}) and (\ref{eq:Fun2}) with $h(x) = (x - \mu)^2$.
\end{proof}

Finally, from Theorem \ref{thm:MSE}, we derive the following universal bound for the MSE.

\begin{thm}
\begin{equation}
\mathbb{E}\left[\left(\widehat{\mu}_N - \mu\right)^2\right] \leq \sigma^2 \left(\sum_{i=1}^L \frac{1}{m_i} + \frac{L+1}{n}\right).\label{eq:BoundMSE}
\end{equation}
In particular, the MSE vanishes if, for fixed $L$, $m_1 \to \infty$ (and hence $\forall i : m_i \rightarrow \infty$ and $n \rightarrow \infty$).
\end{thm}

\begin{proof}
This is derived from Theorem \ref{thm:MSE} in the same way as Theorem \ref{thm:BiasBound} was derived from Theorem \ref{thm:Bias}.
\end{proof}

We wish to conclude this section with the following remarks:
\begin{enumerate}
\item The bounds (\ref{eq:BiasBound}) and (\ref{eq:BoundMSE}) hold for the generic stopping rule described by (a), (b), and (c) in section 1, which contains many classical stopping rules as a special case. Therefore, both bounds have a wide range of applicability.

\item The fact that $0 < m_1 < m_2 < \ldots < m_L < n$ allows us to derive from (\ref{eq:BiasBound}) that
$$\left|\mathbb{E}[\widehat{\mu}_N - \mu]\right| \leq \frac{2 \sigma L \sqrt{\frac{2}{\pi}}}{\sqrt{m_1}}.$$
That is, for fixed $L$, the bias converges to $0$ as $m_1 \rightarrow \infty$ at least with rate $1/\sqrt{m_1}$. Moreover, this rate is optimal. Indeed, taking $\mu = 0$, $\sigma = 1$, $L = 1$, $m_1 = m$, $n = 2m$, and $\psi_m(x) = 1_{\mathbb{R}^+}$, the characteristic function of the set $\mathbb{R}^+$, leads to a trial with maximal length $2m$, in which one interim analysis is performed at $m$. The trial is stopped if $K_m \geq 0$, and continued otherwise. In this case, for independent $\xi$ and $\eta$ with law $N(0,1)$,
\begin{equation*}
\mathbb{E}\left[\xi \psi_m(\sqrt{m} \xi)\right] = \int_0^\infty u \phi(u) du = \phi(0) = \frac{1}{\sqrt{2 \pi}}\label{eq:ExampleC1}
\end{equation*}
and
\begin{equation*}
\mathbb{E}\left[\eta \psi_m\left(\sqrt{\frac{m}{2}} \xi + \sqrt{\frac{m}{2}} \eta\right)\right] = \mathbb{E} \left[\int_{- \xi}^\infty u \phi(u) du\right] = \mathbb{E}[\phi(\xi)] = \int_{-\infty}^\infty \phi^2(u) du = \frac{1}{2\sqrt{\pi}},\label{eq:ExampleC2}
\end{equation*}
from which we deduce that (\ref{eq:Bias}) reduces to
\begin{equation*}
\mathbb{E}\left[\widehat{\mu}_N - \mu\right] = \frac{1}{2 \sqrt{2 \pi}} \frac{1}{\sqrt{m}}.
\end{equation*}
\item The fact that $0 < m_1 < m_2 < \ldots < m_L < n$ allows us to derive from (\ref{eq:MSE}) that 
$$\mathbb{E}\left[\left(\widehat{\mu}_N - \mu\right)^2\right] \leq \frac{\sigma^2 (2 L + 1)}{m_1}.$$
That is, for fixed $L$, the MSE converges to $0$ as $m_1 \rightarrow \infty$ at least with rate $1/m_1$. This rate is again optimal. Indeed, take, as in the previous remark, $\mu = 0$, $\sigma = 1$, $L = 1$, $m_1 = m$, $n = 2m$, and $\psi_m(x) = 1_{\mathbb{R}^+}$. Then, for independent $\xi$ and $\eta$ with law $N(0,1)$,
\begin{equation*}
\mathbb{E}\left[\xi^2 \psi_m(\sqrt{m} \xi)\right] = \int_0^\infty u^2 \phi(u) du = \frac{1}{2}
\end{equation*}
and
\begin{eqnarray*}
\lefteqn{\mathbb{E}\left[\eta^2 \psi_m \left(\sqrt{\frac{m}{2}} \xi + \sqrt{\frac{m}{2}} \eta\right)\right]}\\ 
&=& \mathbb{E} \left[\int_{- \xi}^\infty u^2 \phi(u) du\right] = \mathbb{E}[\Phi(\xi)] + \mathbb{E}[\xi\phi(\xi)] = \int_{-\infty}^\infty \phi(u) \Phi(u) du + \int_{-\infty}^\infty u \phi^2(u) du = \frac{1}{2},
 \end{eqnarray*}
 which shows that (\ref{eq:MSE}) becomes
 \begin{equation*}
 \mathbb{E}\left[\left(\widehat{\mu}_N - \mu\right)^2\right]  =  \frac{3}{4m}.
 \end{equation*}

\item One frequently encounters a group sequential trial in which an interim analysis is performed after every $m$ observations, with $m \in \mathbb{N}_0$ fixed. In our setting, this corresponds to the choices $m_i = im$, where $i \in \{1,\ldots,L\}$, and $n = (L+1)m$. In this case, the bound (\ref{eq:BoundMSE}) reduces to
\begin{equation}
\mathbb{E}\left[\left(\widehat{\mu}_N - \mu\right)^2\right] \leq  \frac{\sigma^2}{m} \left(1 + \sum_{i=1}^L \frac{1}{i}\right) \leq \frac{\sigma^2}{m} \left(2 + \log(L)\right),\label{eq:LogG}
\end{equation}
where $\log$ is the natural logarithm and the last inequality follows by
$$\sum_{i=1}^L \frac{1}{i} = 1 + \sum_{i=2}^L \frac{1}{i} \leq 1 + \int_1^L \frac{dx}{x} = 1 + \log(L),$$
which is seen by comparing on $[1,\infty[$ the graph of the map $y = 1/x$ with the graph of the map that constantly takes the value $1/i$ on $[i-1,i]$, where $i \in \{2,3,\ldots\}$. Taking e.g. $\sigma = 1$, $m = 40$, and $L = 9$, corresponds to a trial of maximal length $400$ in which interim analyses are performed after every $40$ observations. In this case, (\ref{eq:LogG}) gives
$$\mathbb{E}\left[\left(\widehat{\mu}_N - \mu\right)^2\right] \leq \frac{1}{40} (2 + \log(9)) \approx 0.105.$$

\item Again in a trial in which interim analyses are performed after every $m$ observations, the inequality
$$\mathbb{E}\left[\left|\widehat{\mu}_N - \mu\right|\right] \leq \left(\mathbb{E}\left[\left(\widehat{\mu}_N - \mu\right)^2\right]\right)^{1/2}$$
allows us to derive from (\ref{eq:LogG}) the following bound for the bias:
\begin{equation}
\left|\mathbb{E}[\widehat{\mu}_N - \mu]\right| \leq  \sigma \sqrt{\frac{2 + \log(L)}{m}}.\label{eq:LogGB}
\end{equation}

\item Our results show that it is beneficial to start with a sufficiently large first contingent. The bias and the MSE are then generally acceptably small, even when gauged through the uniform bounds (\ref{eq:BiasBound}) and (\ref{eq:BoundMSE}). For specific stopping rules, results may be much sharper, as will be illustrated by our simulation study in section \ref{sec:Simulations}. The question may arise as to whether it is ethical to expose a relatively large first contingent. However, this issue should be approached cautiously. One should consider the expected trial length; designs should be chosen by concentrating on this quantity, rather than on the minimal length. Indeed, a very small minimal length, combined with a very low probability for this to be realized in a given study, is of little value.
\end{enumerate}

\section{Asymptotic normality and confidence intervals}\label{sec:AsymNor}

In this section, we will establish asymptotic normality in total variation distance for $\widehat{\mu}_N$. We need the following lemma.

\begin{lem}
Let $\xi$ and $\eta$ be independent random variables with law $N(0,1)$. Then, for a bounded Borel measurable map $f : \mathbb{R} \rightarrow \mathbb{R}$,
\begin{eqnarray}
\lefteqn{\mathbb{E}\left[f\left(\frac{\sqrt{N}}{\sigma} \left(\widehat{\mu}_N - \mu\right)\right)\right] = \mathbb{E}[f(\xi)]}\label{eq:dtvlem}\\
&&+ \sum_{i=1}^L\mathbb{E}\left[f(\xi) \left(\left(\psi_{m_i} \mathcal{N}_i\right)\left(\mu m_i + \sigma \sqrt{m_i} \xi\right) - \left(\psi_{m_i} \mathcal{N}_i\right)\left(\mu m_i + \sigma \sqrt{\frac{m_i(n - m_i)}{n}} \xi + \sigma \frac{m_i}{\sqrt{n}} \eta\right)\right) \right].\nonumber
\end{eqnarray}
\end{lem}

\begin{proof}
We have
\begin{eqnarray*}
\mathbb{E}\left[f\left(\frac{\sqrt{N}}{\sigma} \left(\widehat{\mu}_N - \mu\right)\right)\right]
= \sum_{i=1}^L \mathbb{E}\left[f\left(\frac{\sqrt{m_i}}{\sigma} \left(\widehat{\mu}_{N} - \mu\right)\right)1_{\{N = m_i\}}\right] + \mathbb{E}\left[f\left(\frac{\sqrt{n}}{\sigma} \left(\widehat{\mu}_N - \mu\right)\right) 1_{\{N = n\}}\right]. 
\end{eqnarray*}
Now apply (\ref{eq:Fun1}) with $h(x) = f\left(\frac{\sqrt{m_i}}{\sigma} \left(x - \mu\right)\right)$ and (\ref{eq:Fun2}) with $h(x) = f\left(\frac{\sqrt{n}}{\sigma} \left(x - \mu\right)\right)$. This gives (\ref{eq:dtvlem}).
\end{proof}

Recall that the total variation distance between random variables $\zeta_1$ and $\zeta_2$ is given by
$$d_{TV}(\zeta_1,\zeta_2) = \sup_{A} \left|\mathbb{P}[\zeta_1 \in A] - \mathbb{P}[\zeta_2 \in A]\right|,$$
the supremum running over all Borel measurable sets $A \subset \mathbb{R}$. It is well known that convergence in total variation distance implies weak convergence, but that the converse generally fails to hold.

The following result provides an explicit bound for the total variation distance between the standard normal distribution and the law of the quantity $\frac{\sqrt{N}}{\sigma} \left(\widehat{\mu}_N - \mu\right)$.

\begin{thm}\label{thm:dtvbd}
Let $\xi$ and $\eta$ be independent random variables with law $N(0,1)$. Then
\begin{eqnarray}
\lefteqn{d_{TV}\left(N(0,1),\frac{\sqrt{N}}{\sigma} \left(\widehat{\mu}_N - \mu\right)\right)}\label{eq:dtvbound}\\
&&\leq \sum_{i=1}^L \mathbb{E}\left[ \left|\left(\psi_{m_i} \mathcal{N}_i\right)\left(\mu m_i + \sigma \sqrt{m_i} \xi\right) - \left(\psi_{m_i} \mathcal{N}_i\right)\left(\mu m_i + \sigma \sqrt{\frac{m_i(n - m_i)}{n}} \xi + \sigma \frac{m_i}{\sqrt{n}} \eta\right)\right| \right].\nonumber
\end{eqnarray}
In particular, if the $\psi_{m_i}$ are continuous, for fixed $L$ and $m_1,\ldots,m_L$,
\begin{equation}
\lim_{n \rightarrow \infty} d_{TV}\left(N(0,1),\frac{\sqrt{N}}{\sigma} \left(\widehat{\mu}_N - \mu\right)\right) = 0.\label{eq:dtvlim}
\end{equation}
\end{thm}

\begin{proof}
Fix a Borel measurable set $A \subset \mathbb{R}$. Applying (\ref{eq:dtvlem}) with $f(x) = 1_A(x)$, gives
\begin{eqnarray*}
\lefteqn{\left|\mathbb{E}\left[1_A\left(\frac{\sqrt{N}}{\sigma} \left(\widehat{\mu}_N - \mu\right)\right) - 1_A(\xi)\right]\right|}\\
&=& \left| \sum_{i=1}^L\mathbb{E}\left[1_A(\xi) \left(\left(\psi_{m_i} \mathcal{N}_i\right)\left(\mu m_i + \sigma \sqrt{m_i} \xi\right) - \left(\psi_{m_i} \mathcal{N}_i\right)\left(\mu m_i + \sigma \sqrt{\frac{m_i(n - m_i)}{n}} \xi + \sigma \frac{m_i}{\sqrt{n}} \eta\right)\right) \right]\right|\\
&\leq& \sum_{i=1}^L \mathbb{E}\left[ \left|\left(\psi_{m_i} \mathcal{N}_i\right)\left(\mu m_i + \sigma \sqrt{m_i} \xi\right) - \left(\psi_{m_i} \mathcal{N}_i\right)\left(\mu m_i + \sigma \sqrt{\frac{m_i(n - m_i)}{n}} \xi + \sigma \frac{m_i}{\sqrt{n}} \eta\right)\right| \right],
\end{eqnarray*}
entailing (\ref{eq:dtvbound}).

Furthermore, suppose that the $\psi_{m_i}$ are continuous. Using Theorem \ref{thm:RecursiveDescriptionNormalTransform}, it is easily checked that the $\mathcal{N}_i$ are also continuous. Hence, for fixed $L$ and $m_1,\ldots,m_L$, $\left(\psi_{m_i} \mathcal{N}_i\right)\left(\mu m_i + \sigma \sqrt{\frac{m_i(n - m_i)}{n}} \xi + \sigma \frac{m_i}{\sqrt{n}} \eta\right)_n$ tends to $\left(\psi_{m_i} \mathcal{N}_i\right)\left(\mu m_i + \sigma \sqrt{m_i} \xi\right)$ pointwise. Thus, each $\psi_{m_i} \mathcal{N}_i$ taking values between $0$ and $1$ by Theorem \ref{lem:NBetween0And1}, the Lebesgue dominated convergence theorem allows us to derive (\ref{eq:dtvlim}) from (\ref{eq:dtvbound}).
\end{proof}

From Theorem \ref{thm:dtvbd}, we easily derive the following conclusion for naive confidence intervals based on $\widehat{\mu}_N$.

\begin{thm}\label{thm:CI}
Let $\Phi$ be the standard normal cumulative distribution function, and $\xi$ and $\eta$ independent random variables with law $N(0,1)$. Then, for $x \in \mathbb{R}$,
\begin{eqnarray}
\lefteqn{\left|2 \Phi (x) - 1 - \mathbb{P}\left[\widehat{\mu}_{N} - \frac{\sigma}{\sqrt{N}} x \leq \mu \leq \widehat{\mu}_{N} + \frac{\sigma}{\sqrt{N}} x \right]\right|}\label{eq:CIBD}\\
&\leq& \sum_{i=1}^L \mathbb{E}\left[ \left|\left(\psi_{m_i} \mathcal{N}_i\right)\left(\mu m_i + \sigma \sqrt{m_i} \xi\right) - \left(\psi_{m_i} \mathcal{N}_i\right)\left(\mu m_i + \sigma \sqrt{\frac{m_i(n - m_i)}{n}} \xi + \sigma \frac{m_i}{\sqrt{n}} \eta\right)\right| \right],\nonumber
\end{eqnarray}
which, for fixed $L$ and $m_1,\ldots,m_L$, tends to $0$ if $n \to \infty$, provided that the $\psi_{m_i}$ are continuous.
\end{thm}

\begin{proof}
This follows from Theorem \ref{thm:dtvbd} by considering the Borel set $A = [-x,x]$.
\end{proof}

We conclude this section with the following remarks:
\begin{enumerate}
	\item Notice the surprising fact that the upper bound in (\ref{eq:dtvbound}) and (\ref{eq:CIBD}) vanishes if $n \to \infty$ for all fixed choices of $L$ and $m_1, \ldots, m_L$. In particular, contrary to the upper bounds in (\ref{eq:BiasBound}) for the bias and in (\ref{eq:BoundMSE}) for the MSE, in studies with large maximal length $n$, the upper bound in (\ref{eq:dtvbound}) and (\ref{eq:CIBD}) always vanishes, even if the $m_i$ are small, i.e. if the interim analyses are performed early.
	
	\item  The bound (\ref{eq:CIBD}) justifies the use of naive confidence intervals based on $\widehat{\mu}_N$, provided that the $m_i$ are kept fixed, the $\psi_{m_i}$ are continuous, and the maximal length $n$ of the trial is large enough. We wish to point out that our conclusion for confidence intervals is less powerful than our statements for bias and MSE in the previous section. Indeed, we have not provided a rate of convergence. A deeper study of confidence intervals in a group sequential trial setting turns out to be much harder, and will be treated in subsequent work.
\end{enumerate}

\section{Connections with likelihood theory}\label{sec:LHT}

In this section, we will connect the theory developed in the previous sections to marginal and conditional maximum likelihood estimation after a group sequential trial. As a starting point for likelihood theoretic arguments, we assume that the distribution of the $X_i$ comes from the parametric family $\{N(\theta,\sigma^2) \mid \theta \in \mathbb{R}\}$.

Since each of the $X_i$ has distribution $N(\theta,\sigma^2)$, we observe that the joint density of the $X_i$ gathered in the trial is given by 
$$f_{X_1,\ldots,X_N}(\theta, x_1,\ldots,x_N) = \frac{1}{\sigma^{N}}\prod_{j=1}^{N} \phi\left(\frac{x_j - \theta}{\sigma}\right).$$
Therefore, classical likelihood theory kicks in and we conclude that the marginal maximum likelihood estimator (MLE) for $\mu$ is the ordinary sample mean $\widehat{\mu}_N$. In the previous sections, we have provided evidence of the fact that this estimator performs well in terms of bias and MSE if the first interim analysis is performed not too early, and in terms of asymptotic normality and confidence intervals if the maximal length of the study is large enough.

We now turn to conditional maximum likelihood estimation (CMLE) after a group sequential trial. More precisely, we will link the CMLE for $\mu$, conditioned on $N$, to the sample average, from which it will follow that it coincides with the `conditional bias reduction estimate', studied in \cite{FDL00}, section 3.3. We will use the following lemma, which lies at the heart of Stein's method (\cite{CGS11}).

\begin{lem}
Let $\eta$ be a random variable with law $N(0,1)$, $g : \mathbb{R} \rightarrow \mathbb{R}$ a Borel measurable map with $\mathbb{E}[\left|g(\eta)\right|] < \infty$, $A \in \mathbb{R}$, and $B \in \mathbb{R}_0$. Then the map 
$$\theta \mapsto \mathbb{E}[g(A\theta +B \eta)], \phantom{1} \theta \in \mathbb{R},$$ 
is smooth. Furthermore,
\begin{equation}
\frac{d}{d\theta} \mathbb{E}[g(A\theta + B\eta)] = \frac{A}{B}\mathbb{E}[\eta g(A \theta + B\eta)]\label{eq:Stein1}
\end{equation}
and 
\begin{equation}
\frac{d}{d\theta} \mathbb{E}[\eta g(A\theta + B\eta)] = \frac{A}{B}\mathbb{E}[(\eta^2 - 1) g(A \theta + B \eta)].\label{eq:Stein2}
\end{equation}
\end{lem}

\begin{proof}
We have 
$$\frac{d}{d\theta} \mathbb{E}[g(A \theta + B \eta)] = \frac{d}{d\theta} \int_{-\infty}^\infty g(A \theta + B u) \phi(u) du,$$
which, performing the change of variables $t = A \theta + B u$, 
\begin{eqnarray*}
=  \frac{1}{B} \frac{d}{d\theta} \int_{-\infty}^\infty g(t) \phi\left(\frac{t - A \theta}{B}\right) dt &=&  \frac{1}{B} \int_{-\infty}^\infty g(t) \frac{d}{d\theta} \phi\left(\frac{t - A \theta}{B}\right) dt\\ &=& \frac{A}{B^2} \int_{-\infty}^\infty g(t) \left(\frac{t - A \theta}{B} \right) \phi\left(\frac{t - A \theta}{B}\right) dt,
\end{eqnarray*}
which, performing the change of variables $v = \frac{t - A \theta}{B}$,
$$= \frac{A}{B} \int_{- \infty}^\infty v g(A \theta + B v) \phi(v) dv = \frac{A}{B}\mathbb{E}[\eta g(A \theta + B\eta)].$$
This proves (\ref{eq:Stein1}). The proof of (\ref{eq:Stein2}) is analogous.
\end{proof}

Now let $\xi$ and $\eta$ be independent random variables with law $N(0,1)$, and suppose that, for each $\theta$, $N$ can take each of the values $m_1,\ldots,m_L,n$ with a strictly positive probability, i.e. the expressions (\ref{eq:PNIsmi}) and (\ref{eq:PNIsn}) are nonzero if $\mu$ is replaced by $\theta$. Then, using Bayes' Theorem, the fact that the $X_i$ have law $N(\theta,\sigma^2)$, and plugging in (c) in section 1, and (\ref{eq:PNIsmi}) and (\ref{eq:PNIsn}) with $\mu$ replaced by $\theta$, it holds for the conditional likelihood of the $X_i$ given $N$ that, for $i \in \{1,\ldots,L\}$,
\begin{eqnarray}
\lefteqn{f_{X_1,\ldots,X_{m_i} \mid N} (\theta,x_1,\ldots,x_{m_i} \mid m_i)}\\
&=& \frac{1}{\sigma^{m_i}}\prod_{j=1}^{m_i} \phi\left(\frac{x_j - \theta}{\sigma}\right) \frac{\psi_{m_i}(\sum_{k=1}^{m_i} x_k) \prod_{j=1}^{i-1} \left[1 - \psi_{m_j}\left(\sum_{k=1}^{m_j} x_k\right)\right]}{\mathbb{E}\left[\left(\psi_{m_i} \mathcal{N}_i\right)(\theta m_i + \sigma \sqrt{m_i} \xi)\right]},\nonumber
\end{eqnarray}
and
\begin{eqnarray}
\lefteqn{f_{X_1,\ldots,X_{n} \mid N} (\theta,x_1,\ldots,x_{n} \mid n)}\\
&=& \frac{1}{\sigma^{n}}\prod_{j=1}^{n} \phi\left(\frac{x_j - \theta}{\sigma}\right)\frac{1 - \sum_{i=1}^L \psi_{m_i}(\sum_{k=1}^{m_i} x_k) \prod_{j=1}^{i-1} \left[1 - \psi_{m_j}\left(\sum_{k=1}^{m_j} x_k\right)\right]}{1 - \sum_{i=1}^L \mathbb{E}\left[\left(\psi_{m_i} \mathcal{N}_i\right)\left(\theta m_i + \sigma \sqrt{\frac{m_i(n - m_i)}{n}} \xi + \sigma \frac{m_i}{\sqrt{n}} \eta\right)\right]}.\nonumber
\end{eqnarray}
In particular, up to an additive constant not depending on $\theta$, the conditional log-likelihood is given by, for $i \in \{1,\ldots,L\}$,
\begin{equation}
\mathcal{L}(\theta,x_1,\ldots,x_{m_i} \mid m_i) = - \frac{1}{2 \sigma^2} \sum_{j=1}^{m_i} \left(x_j - \theta\right)^2 - \log \mathbb{E}\left[\left(\psi_{m_i} \mathcal{N}_i\right)(\theta m_i + \sigma \sqrt{m_i} \xi)\right],\label{eq:LL1}
\end{equation}
and
\begin{eqnarray}
\lefteqn{\mathcal{L}(\theta,x_1,\ldots,x_{n} \mid n)}\label{eq:LL2}\\
&=& - \frac{1}{2 \sigma^2} \sum_{j=1}^{n} \left(x_j - \theta\right)^2 - \log\left(1 - \sum_{i=1}^L \mathbb{E}\left[\left(\psi_{m_i} \mathcal{N}_i\right)\left(\theta m_i + \sigma \sqrt{\frac{m_i(n - m_i)}{n}} \xi + \sigma \frac{m_i}{\sqrt{n}} \eta\right)\right]\right).\nonumber
\end{eqnarray}
Applying (\ref{eq:Stein1}) with $g = \psi_{m_i} \mathcal{N}_i$, $A = m_i$, and $B = \sigma \sqrt{m_i}$, shows that the partial derivative of (\ref{eq:LL1}) with respect to $\theta$ is 
\begin{equation}
\frac{\partial}{\partial \theta}\mathcal{L}(\theta,x_1,\ldots,x_{m_i} \mid m_i)
= \frac{1}{\sigma^2} \sum_{j=1}^{m_i} x_j - \frac{m_i}{\sigma^2} \theta - \frac{\sqrt{m_i}}{\sigma} \frac{\mathbb{E}[\xi \left(\psi_{m_i} \mathcal{N}_i\right)(\theta m_i + \sigma \sqrt{m_i} \xi) ]}{\mathbb{E}\left[\left(\psi_{m_i} \mathcal{N}_i\right)(\theta m_i + \sigma \sqrt{m_i} \xi)\right]}, \label{eq:LL1D}
\end{equation}
and, applying (\ref{eq:Stein1}) with $g(\theta) = \mathbb{E}\left[\left(\psi_{m_i} \mathcal{N}_i\right)\left(\theta + \sigma \sqrt{\frac{m_i(n - m_i)}{n}} \xi\right) \right]$, $A = m_i$, and $B = \sigma \frac{m_i}{\sqrt{n}}$, and using independence of $\xi$ and $\eta$, reveals that the partial derivative of (\ref{eq:LL2}) with respect to $\theta$ is
\begin{eqnarray}
\lefteqn{\frac{\partial}{\partial \theta}\mathcal{L}(\theta,x_1,\ldots,x_{n} \mid n)}\label{eq:LL2D}\\
&=& \frac{1}{\sigma^2} \sum_{j=1}^n x_j - \frac{n}{\sigma^2} \theta + \frac{\sqrt{n}}{\sigma}\frac{\sum_{i=1}^L \mathbb{E}\left[\eta \left(\psi_{m_i} \mathcal{N}_i\right)\left(\theta m_i + \sigma \sqrt{\frac{m_i(n - m_i)}{n}} \xi + \sigma \frac{m_i}{\sqrt{n}} \eta\right)\right]}{1 - \sum_{i=1}^L \mathbb{E}\left[\left(\psi_{m_i} \mathcal{N}_i\right)\left(\theta m_i + \sigma \sqrt{\frac{m_i(n - m_i)}{n}} \xi + \sigma \frac{m_i}{\sqrt{n}} \eta\right)\right]}.\nonumber
\end{eqnarray}
Furthermore, using (\ref{eq:Fun1}) with $h(x) = x - \theta$ and (\ref{eq:PNIsmi}) with $\mu$ replaced by $\theta$, (\ref{eq:LL1D}) leads to
\begin{eqnarray}
\lefteqn{\frac{\partial}{\partial \theta}\mathcal{L}(\theta,X_1,\ldots,X_{m_i} \mid m_i)}\label{eq:LL1DS}\\
&=& \frac{m_i}{\sigma^2} \left(\widehat{\mu}_{m_i}  - \theta - \mathbb{E}_\theta\left[\left(\widehat{\mu}_N -\theta\right) \mid N = m_i\right]\right)\nonumber\\
&=& \frac{m_i}{\sigma^2} \left(\widehat{\mu}_{m_i} - \mathbb{E}_\theta\left[\widehat{\mu}_N \mid N = m_i\right]\right),\nonumber
\end{eqnarray}
and, using (\ref{eq:Fun2}) with $h(x) = x - \theta$ and (\ref{eq:PNIsn}) with $\mu$ replaced by $\theta$, (\ref{eq:LL2D}) gives
\begin{eqnarray}
\lefteqn{\frac{\partial}{\partial \theta}\mathcal{L}(\theta,X_1,\ldots,X_{n} \mid n)}\label{eq:LL2DS}\\
&=& \frac{n}{\sigma^2} \left(\widehat{\mu}_{n}  - \theta - \mathbb{E}_\theta\left[\left(\widehat{\mu}_N - \theta\right) \mid N = n\right]\right)\nonumber\\
&=& \frac{n}{\sigma^2} \left(\widehat{\mu}_{n}  - \mathbb{E}_\theta\left[\widehat{\mu}_N  \mid N = n\right]\right).\nonumber
\end{eqnarray}
Also, applying (\ref{eq:Stein2}) with $g = \psi_{m_i} \mathcal{N}_i$, $A = m_i$, and $B = \sigma \sqrt{m_i}$, shows that the partial derivative of (\ref{eq:LL1D}) with respect to $\theta$ is
\begin{eqnarray}
\lefteqn{\frac{\partial^2}{\partial \theta^2}\mathcal{L}(\theta,x_1,\ldots,x_{m_i} \mid m_i)}\label{eq:LLDD1}\\
&=& - \frac{m_i}{\sigma^2}\left( \frac{\mathbb{E}\left[\xi^2 \left(\psi_{m_i} \mathcal{N}_i\right)(\theta m_i + \sigma \sqrt{m_i} \xi) \right]}{\mathbb{E}[\left(\psi_{m_i} \mathcal{N}_i\right)(\theta m_i + \sigma \sqrt{m_i} \xi) ]} - \left(\frac{\mathbb{E}\left[\xi \left(\psi_{m_i} \mathcal{N}_i\right)(\theta m_i + \sigma \sqrt{m_i} \xi)\right]}{\mathbb{E}\left[\left(\psi_{m_i} \mathcal{N}_i\right)(\theta m_i + \sigma \sqrt{m_i} \xi)\right]}\right)^2\right),\nonumber
\end{eqnarray} 
and, applying (\ref{eq:Stein2}) with $g(\theta) = \mathbb{E}\left[\left(\psi_{m_i} \mathcal{N}_i\right)\left(\theta + \sigma \sqrt{\frac{m_i(n - m_i)}{n}} \xi\right) \right]$, $A = m_i$, and $B = \sigma \frac{m_i}{\sqrt{n}}$, and using independence of $\xi$ and $\eta$, shows that the partial derivative of (\ref{eq:LL2D}) with respect to $\theta$ is
\begin{eqnarray}
\lefteqn{\frac{\partial^2}{\partial \theta^2}\mathcal{L}(\theta,x_1,\ldots,x_{n} \mid n)}\label{eq:LLDD2}\\
&=& \frac{n}{\sigma^2}\left(\frac{ \sum_{i=1}^L \mathbb{E}\left[\eta^2 \left(\psi_{m_i} \mathcal{N}_i\right)\left(\theta m_i + \sigma \sqrt{\frac{m_i(n - m_i)}{n}} \xi + \sigma \frac{m_i}{\sqrt{n}} \eta\right) \right]}{1 -  \sum_{i=1}^L \mathbb{E}\left[\left(\psi_{m_i} \mathcal{N}_i\right)\left(\theta m_i + \sigma \sqrt{\frac{m_i(n - m_i)}{n}} \xi + \sigma \frac{m_i}{\sqrt{n}} \eta\right)\right]}\right)\nonumber\\
&& - \frac{n}{\sigma^2} \left(\frac{\sum_{i=1}^L\mathbb{E}\left[\eta \left(\psi_{m_i} \mathcal{N}_i\right)\left(\theta m_i + \sigma \sqrt{\frac{m_i(n - m_i)}{n}} \xi + \sigma \frac{m_i}{\sqrt{n}} \eta\right) \right]}{1 - \sum_{i=1}^L\mathbb{E}\left[\left(\psi_{m_i} \mathcal{N}_i\right)\left(\theta m_i + \sigma \sqrt{\frac{m_i(n - m_i)}{n}} \xi + \sigma \frac{m_i}{\sqrt{n}} \eta\right) \right]}\right)^2.\nonumber
\end{eqnarray}
Furthermore, using (\ref{eq:Fun1}), with respectively $h_1(x) = (x - \theta)^2$ and $h_2(x) = x - \theta$, and (\ref{eq:PNIsmi}), each time with $\mu$ replaced by $\theta$, transforms (\ref{eq:LLDD1}) into
\begin{eqnarray}
\lefteqn{\frac{\partial^2}{\partial \theta^2}\mathcal{L}(\theta,X_1,\ldots,X_{m_i} \mid m_i)}\label{eq:LLDD1S}\\
&=& -\frac{m_i^2}{\sigma^4} \left(\mathbb{E}_\theta\left[\left(\widehat{\mu}_N - \theta\right)^2 \mid N = m_i\right] - \left(\mathbb{E}_\theta\left[\left(\widehat{\mu}_N - \theta\right) \mid N = m_i\right]\right)^2\right)\nonumber\\
&=& - \frac{m_i^2}{\sigma^4} {\text{\upshape Var}}_\theta[\widehat{\mu}_N \mid N = m_i],\nonumber
\end{eqnarray}
and, using (\ref{eq:Fun2}), with respectively $h_1(x) = (x - \theta)^2$ and $h_2(x) = x - \theta$, and (\ref{eq:PNIsn}), each time with $\mu$ replaced by $\theta$, shows that (\ref{eq:LLDD2}) gives
\begin{eqnarray}
\lefteqn{\frac{\partial^2}{\partial \theta^2}\mathcal{L}(\theta,X_1,\ldots,X_{n} \mid n)}\label{eq:LLDD2S}\\
&=& -\frac{n^2}{\sigma^4} \left(\mathbb{E}_\theta\left[\left(\widehat{\mu}_N - \theta\right)^2 \mid N = n\right] - \left(\mathbb{E}_\theta\left[\left(\widehat{\mu}_N - \theta\right) \mid N = n\right]\right)^2\right)\nonumber\\
&=& - \frac{n^2}{\sigma^4} {\text{\upshape Var}}_\theta\left[\widehat{\mu}_N \mid N = n\right].\nonumber
\end{eqnarray}

The relations (\ref{eq:LL1DS}), (\ref{eq:LL2DS}), (\ref{eq:LLDD1S}), and (\ref{eq:LLDD2S}) are summarized by stating that the conditional log-likelihood satisfies the equations
 \begin{equation}
\frac{\partial}{\partial \theta}\mathcal{L}\left(\theta,X_1,\ldots,X_{N} \mid N\right) = \frac{N}{\sigma^2} \left(\widehat{\mu}_{N} -  \mathbb{E}_\theta\left[\widehat{\mu}_N\mid N \right]\right)\label{eq:LLDS}
\end{equation}
and
\begin{equation}
\frac{\partial^2}{\partial \theta^2}\mathcal{L}(\theta,X_1,\ldots,X_{N} \mid N) = - \frac{N^2}{\sigma^4} {\text{\upshape Var}}_\theta\left[\widehat{\mu}_N \mid N\right].\label{eq:LLDDS}
\end{equation}
We conclude that the maximum likelihood estimator conditioned on $N$, denoted by $\widehat{\mu}_{c,N}$, satisfies the relation
\begin{equation}
\widehat{\mu}_N = \mathbb{E}_{\widehat{\mu}_{c,N}}[\widehat{\mu}_N \mid N].\label{eq:MUC}
\end{equation}

It follows from (\ref{eq:MUC}) that $\widehat{\mu}_{c,N}$ coincides with the `conditional bias reduction estimate', studied in \cite{FDL00}, section 3.3. In section 3.4 of that paper, one provides empirical evidence of the fact that $\widehat{\mu}_{c,N}$ outperforms $\widehat{\mu}_N$ in trials that are stopped early. Combining this information with our results obtained in the previous sections, it seems plausible to recommend $\widehat{\mu}_N$ if the first interim analysis is performed not too early, and $\widehat{\mu}_{c,N}$ otherwise.

\section{Simulations}\label{sec:Simulations}

To illustrate our findings, a simulation study was conducted to investigate the speed of convergence. 
Two different cases were considered: continuous normal, $X_1, X_2, \ldots, X_n$ i.i.d.\ $\sim N(\mu, 1)$, and discrete Bernoulli, $X_1, X_2, \ldots, X_n$ i.i.d.\ $\sim B(\pi)$, with different choices for the parameter values. For each case, the following design choices were made: 1000 random samples were generated, each of size $n$. To every sample, several stopping conditions were applied: (1) no stopping; (2) one stopping occasion ($m_1$) using the K criterion; (3)  1 stopping occasion ($m_1$) using the probit criterion; (4) 3 stopping occasions ($m_1$, $m_2$, or $m_3$) using the K criterion; and (5) 3 stopping occasions ($m_1$, $m_2$, or $m_3$) using the probit criterion. All stopping criteria were applied to the sample mean $K_{m_i}=\frac{1}{m_i}\sum_{i=1}^{m_i}{X_i}$ in the following way. For the K criterion, stop at $m_i$, if $K_{m_i} < 0$. For the probit criterion, first, the probability $\Phi\left(\alpha + \beta K_{m_i} \right)$ was calculated and a random uniform vector $U \sim \mbox{Uniform~(0,1)}$ was generated; if $U \leq \Phi\left(\alpha + \beta K_{m_i} \right)$, then stop at $m_i$, otherwise continue to $m_{i+1}$ (or $n$). All calculations were performed with the R statistical software (R version 3.3.3).

The parameter choice was made as follows: total sample size $n=400$, for probit $\alpha=0$ (kept fixed) and $\beta=-2,-1,0,1,2$. For 1 stopping occasion, $m_1=200$, for 3 stopping occasions, different scenarios were considered: (a) 3 `late' stopping occasions with $m_1=100$, $m_2=200$, $m_3=300$; (b) 3 `early' stopping occasions with $m_1=50$, $m_2=100$, $m_3=150$ and $m_1=25$, $m_2=50$, $m_3=75$; (c) 3 `extremely early' stopping occasions with $m_1=10$, $m_2=20$, $m_3=30$; $m_1=5$, $m_2=10$, $m_3=15$ and $m_1=2$, $m_2=4$, $m_3=6$. For the distribution parameters, the following choice was made. For the normal case, $\mu~=-2,-1,0,1,2$ and the standard deviation is kept fixed  $\sigma=1$; for the Bernoulli case, $\pi=0.1,0.3,0.5,0.7,0.9$ and also some `extreme' values of $\pi$ such as $0.001,0.01$.

For each generated sample, the following statistics were calculated: bias as $\frac{1}{N}\sum_{i=1}^{N}{(\widehat{\theta}-\theta)}$, relative bias as $\frac{\frac{1}{N}\sum_{i=1}^{N}{(\widehat{\theta}-\theta)}}{\theta}$, mean square error (MSE) as $\frac{1}{N}\sum_{i=1}^{N}{(\widehat{\theta}-\theta)^2}$, 95$\%$ confidence interval as an average 95$\%$ confidence interval over all 1000 generated samples, true coverage probability, and average sample size for all 1000 generated samples. All results are summarized in the accompanying supplementary material.

We focused in this study primarily on the behavior of the bias and the MSE. The  simulations conducted confirm the theoretical results for both generic stopping rules in the case of normal target distributions: bias of the sample mean converges to zero with speed $\frac{1}{\sqrt{m_1}}$ and the MSE of the sample mean converges to zero with speed $\frac{1}{m_1}$. In addition, we examined the behavior of the 95$\%$ confidence interval and the finite sample size, and we noted that the coverage probabilities in some cases are not $95 \%$. This is compatible with our results, taking into account the second remark accompanying Theorem \ref{thm:CI}. A deeper study of confidence intervals turns out to be much harder, and will be treated in subsequent work.
 
\section{Concluding remarks}\label{sec:ConRem}

In this paper, we have studied the theory of estimation after a group sequential trial with independent and identically distributed normal outcomes $X_1,X_2,\ldots$ with mean $\mu$ and variance $\sigma^2$. We have denoted the maximal length of the trial as $n$, the places at which the interim analyses of the sum of the outcomes are performed as $0 < m_1 < m_2 < \ldots < m_L < n$, and the actual length of the trial, which is a random variable, as $N$.  At each $m_i$, one decides whether the trial is stopped, i.e. $N = m_i$, or continued, i.e. $N > m_i$. We have based this decision on the generic stopping rule given by (a), (b), and (c) in Section 1, which was shown to contain many classical stopping rules from the literature. Therefore, our setting has a wide range of applicability.

The main goal of this paper is to gain an understanding of the quality of the sample mean $\widehat{\mu}_N = \frac{1}{N} K_N$, where $K_N = \sum_{i=1}^N X_i$, as an estimator for $\mu$ in the above group sequential trial setting. To this end, we have used the normal transform, defined by (\ref{eq:NormalTransform}), as an auxiliary analytic tool to establish the explicit expressions (\ref{eq:JointDensity}) and (\ref{eq:JointDensity2}) for the joint density of $N$ and $K_N$. These expressions were used to obtain formula (\ref{eq:ExL}) for the expected length of the trial, formula (\ref{eq:Bias}) for the bias, and formula (\ref{eq:MSE}) for the MSE. We have derived the upper bound (\ref{eq:BiasBound}) for the bias, which shows that, for fixed $L$, the bias vanishes as $m_1 \to \infty$ at least with rate $1/\sqrt{m_1}$, and the upper bound (\ref{eq:BoundMSE}) for the MSE, entailing that, for fixed $L$, the MSE vanishes as $m_1 \to \infty$ at least with rate $1/m_1$. Both rates were shown to be optimal. Also, for trials in which an interim analysis is performed after every $m$ observations, we have obtained the bound (\ref{eq:LogG}) for the MSE and the bound (\ref{eq:LogGB}) for the bias. Furthermore, we have obtained the upper bound (\ref{eq:dtvbound}) for the total variation distance between $N(0,1)$ and the law of the quantity $\frac{\sqrt{N}}{\sigma} \left(\widehat{\mu}_N - \mu\right)$, which, under a regularity condition, was shown to vanish if, for fixed $L$ and $m_1,\ldots,m_L$, $n \to \infty$. This has also led to the bound (\ref{eq:CIBD}), which, in some cases, justifies the use of naive confidence intervals based on $\widehat{\mu}_N$ if $n$ is large. It is quite surprising that, contrary to the upper bounds in (\ref{eq:BiasBound}) for the bias and in (\ref{eq:BoundMSE}) for the MSE, the upper bound in (\ref{eq:dtvbound}) and (\ref{eq:CIBD}) always vanishes if $n$ is large, even if the $m_i$ are small, i.e. if the interim analyses are performed early. Finally, the theory developed in this paper for the sample mean was shown to fit naturally in the broader framework of maximum likelihood estimation after a group sequential trial. More precisely, the marginal MLE coincides with $\widehat{\mu}_N$, and the conditional MLE $\widehat{\mu}_{c,N}$ satisfies equation (\ref{eq:MUC}), from which we derived that it coincides with the `conditional bias reduction estimate'. Our theoretical findings were illustrated by several simulations.

Based on the obtained results, we suggest that in many realistic cases it is safe to use the ordinary sample mean as a reliable estimator after a group sequential trial.

\section*{Acknowledgements}

We thank the anonymous reviewers for their valuable comments.

\section*{Funding}
Ben Berckmoes is post doctoral fellow at the Fund for Scientific Research of Flanders (FWO).\ \\

Financial support from the IAP research network \#P7/06 of the Belgian Government (Belgian Science Policy) is gratefully acknowledged.

\newpage

\begin{center}
On the sample mean after a group sequential trial\ \\
(Supplementary material)
\end{center}

\begin{center}
\begin{longtable}{llccccccc}
\caption{\em Simulation results for the normal case with the number of simulated i.i.d.\ draws per sample 400, the number of simulated samples 1000. Values for standard deviation and $\alpha$ for the probit rule are kept fixed: $\sigma=1$, $\alpha=0$. CL: 95\% confidence limit, Cov.Prob.: coverage probability. \label{simul_normal}}\\
\hline
\hline
   $\mu$&$\beta$&Bias&Relative Bias& MSE&Lower CL& Upper CL& Cov.Prob.&Aver.Size\\
\hline
\hline
\endfirsthead
\multicolumn{9}{c}%
{\tablename\ \thetable\ -- \it{Continued from previous page}} \\
\hline
\hline
 $\mu$&$\beta$&Bias&Relative Bias& MSE&Lower CL& Upper CL& Cov.Prob.& Aver.Size\\
\hline
\hline
\endhead
\hline \multicolumn{9}{r}{\it{Continued on next page}} \\
\endfoot
\hline
\endlastfoot

\multicolumn{9}{c}{No Stopping} \\
\hline
-2& &~0.000040&-0.00020&0.00235&-2.09764&-1.90157&0.955&400\\
\hline
\multicolumn{9}{c}{1 Stopping, $m=200$, K criterion}\\
\hline
-2& &~0.00148&-0.00074&0.00475&-2.13705&-1.86000&0.961&200\\
\hline
\multicolumn{9}{c}{1 Stopping, $m=200$, probit criterion}\\
\hline
-2&-2&~0.00148&-0.00074&0.00475&-2.13705&-1.86000&0.961&200\\
-2&-1&~0.00128&-0.00064&0.00471&-2.13624&-1.86121&0.962&205\\
-2&~0&~0.00077&-0.00039&0.00361&-2.11833&-1.88012&0.961&296\\
-2&~1&~0.00017&-0.00009&0.00238&-2.09861&-1.90104&0.955&396\\
-2&~2&~0.00040&-0.00020&0.00235&-2.09764&-1.90157&0.955&400\\
\hline
\multicolumn{9}{c}{3 Stoppings, $m=100,200,300$, K criterion}\\
\hline
-2& &-0.00049&~0.00024&0.00946&-2.19627&-1.8047&0.955&100\\
\hline
\multicolumn{9}{c}{3 Stoppings, $m=100,200,300$, probit-criterion}\\
\hline
-2&-2&-0.00049&~0.00024&0.00946&-2.19627&-1.80470&0.955&100\\
-2&-1&-0.00037&~0.00018&0.00943&-2.19473&-1.80601&0.954&102\\
-2&~0&~0.00141&-0.00070&0.00604&-2.15872&-1.83846&0.963&185\\
-2&~1&~0.00037&-0.00019&0.00251&-2.10101&-1.89825&0.958&387\\
-2&~2&~0.00040&-0.00020&0.00235&-2.09764&-1.90157&0.955&400\\
\hline
\multicolumn{9}{c}{No Stopping} \\
\hline
-1& &~0.0004&-0.0004&0.00235&-1.09764&-0.90157&0.955&400\\
\hline
\multicolumn{9}{c}{1 Stopping, $m=200$, K criterion}\\
-1& &~0.00148&-0.00148&0.00475&-1.13705&-0.86000&0.961&200\\
\hline
\multicolumn{9}{c}{1 Stopping, $m=200$, probit criterion}\\
\hline
-1&-2&~0.00118&-0.00118&0.00469&-1.13621&-0.86142&0.962&205\\
-1&-1&~0.00000&~0.00000&0.00433&-1.13168&-0.86831&0.966&234\\
-1&~0&~0.00077&-0.00077&0.00361&-1.11833&-0.88012&0.961&296\\
-1&~1&~0.00092&-0.00092&0.00267&-1.10346&-0.89471&0.959&369\\
-1&~2&~0.00036&-0.00036&0.00239&-1.09851&-0.90076&0.956&396\\
\hline
\multicolumn{9}{c}{3 Stoppings, $m=100,200,300$, K criterion}\\
\hline
-1& &-0.00049&~0.00049&0.00946&-1.19627&-0.8047&0.955&100\\
\hline
\multicolumn{9}{c}{3 Stoppings, $m=100,200,300$, probit criterion}\\
\hline
-1&-2&-0.00112&~0.00112&0.00934&-1.19529&-0.80695&0.955&103\\
-1&-1&-0.00244&~0.00244&0.00841&-1.18774&-0.81715&0.954&120\\
-1&~0&~0.00141&-0.00141&0.00604&-1.15872&-0.83846&0.963&185\\
-1&~1&~0.00335&-0.00335&0.00371&-1.11750&-0.87580&0.955&314\\
-1&~2&~0.00071&-0.00071&0.00253&-1.10069&-0.89789&0.958&387\\
\hline
\multicolumn{9}{c}{No Stopping} \\
\hline
~0& &~0.00040&--&0.00235&-0.09764&~0.09843&0.955&400\\
\hline
\multicolumn{9}{c}{1 Stopping, $m=200$, K criterion}\\
\hline
~0& &-0.01285&--&0.00348&-0.13064&~0.10493&0.961&303\\
\\
\multicolumn{9}{c}{1 Stopping, $m=200$, probit criterion}\\
\hline
~0&-2&-0.00094&--&0.00347&-0.11946&~0.11758&0.964&299\\
~0&-1&~0.00011&--&0.00356&-0.11885&~0.11908&0.963&297\\
~0&~0&~0.00077&--&0.00361&-0.11833&~0.11988&0.961&296\\
~0&~1&~0.00196&--&0.00368&-0.11744&~0.12136&0.962&295\\
~0&~2&~0.00294&--&0.00362&-0.11650&~0.12238&0.963&294\\
\hline
\multicolumn{9}{c}{3 Stoppings, $m=100,200,300$, K criterion}\\
\hline
~0&&-0.03133&--&0.00596&-0.18477&~0.1221&0.958&219\\
\hline
\multicolumn{9}{c}{3 Stoppings, $m=100,200,300$, probit criterion}\\
\hline
~0&-2&-0.00470&--&0.00631&-0.16439&~0.15500&0.958&186\\
~0&-1&-0.00151&--&0.00617&-0.16167&~0.15865&0.959&185\\
~0&~0&~0.00141&--&0.00604&-0.15872&~0.16154&0.963&185\\
~0&~1&~0.00408&--&0.00616&-0.15596&~0.16413&0.960&185\\
~0&~2&~0.00648&--&0.00606&-0.15330&~0.16626&0.965&185\\
\hline
\multicolumn{9}{c}{No Stopping} \\
\hline
~1& &~0.00040&~0.00040&0.00235&~0.90236&~1.09843&0.955&400\\
\hline
\multicolumn{9}{c}{1 Stopping, $m=200$, K criterion}\\
\hline
~1& &~0.00040&~0.00040&0.00235&~0.90236&~1.09843&0.955&400\\
\hline
\multicolumn{9}{c}{1 Stopping, $m=200$, probit criterion}\\
\hline
~1&-2&~0.00009&~0.00009&0.00238&~0.90130&~1.09887&0.955&396\\
~1&-1&-0.00016&-0.00016&0.00268&~0.89550&~1.10419&0.958&369\\
~1&~0&~0.00077&~0.00077&0.00361&~0.88167&~1.11988&0.961&296\\
~1&~1&~0.00128&~0.00128&0.00443&~0.86926&~1.13329&0.966&232\\
~1&~2&~0.00149&~0.00149&0.00473&~0.86388&~1.13910&0.961&204\\
\hline
\multicolumn{9}{c}{3 Stoppings, $m=100,200,300$, K criterion}\\
\hline
~1& &~0.00040&~0.00040&0.00235&~0.90236&~1.09843&0.955&400\\
\hline
\multicolumn{9}{c}{3 Stoppings, $m=100,200,300$, probit criterion}\\
\hline
~1&-2&-0.00057&-0.00057&0.00252&~0.89818&~1.10067&0.958&387\\
~1&-1&-0.00066&-0.00066&0.00360&~0.87856&~1.12011&0.963&315\\
~1&~0&~0.00141&~0.00141&0.00604&~0.84128&~1.16154&0.963&185\\
~1&~1&~0.00067&~0.00067&0.00826&~0.81548&~1.18585&0.957&120\\
~1&~2&~0.00036&~0.00036&0.00935&~0.80628&~1.19444&0.955&103\\
\hline
\multicolumn{9}{c}{No Stopping} \\
\hline
~2& &~0.00040&~0.00020&0.00235&~1.90236&~2.09843&0.955&400\\
\hline
\multicolumn{9}{c}{1 Stopping, $m=200$, K criterion}\\
\hline
~2& &~0.00040&~0.00020&0.00235&~1.90236&~2.09843&0.955&400\\
\hline
\multicolumn{9}{c}{1 Stopping, $m=200$, probit criterion}\\
\hline
~2&-2&~0.00040&~0.00020&0.00235&~1.90236&~2.09843&0.955&400\\
~2&-1&~0.00009&~0.00004&0.00238&~1.90130&~2.09887&0.955&396\\
~2&~0&~0.00077&~0.00039&0.00361&~1.88167&~2.11988&0.961&296\\
~2&~1&~0.00125&~0.00063&0.00470&~1.86373&~2.13877&0.962&205\\
~2&~2&~0.00152&~0.00076&0.00474&~1.86303&~2.14001&0.961&200\\
\hline
\multicolumn{9}{c}{3 Stoppings, $m=100,200,300$, K criterion}\\
\hline
~2& &~0.00040&~0.00020&0.00235&~1.90236&~2.09843&0.955&400\\
\hline
\multicolumn{9}{c}{3 Stoppings, $m=100,200,300$, probit criterion}\\
\hline
~2&-2&~0.00040&~0.00020&0.00235&~1.90236&~2.09843&0.955&400\\
~2&-1&-0.00031&-0.00016&0.00250&~1.89846&~2.10092&0.957&387\\
~2&~0&~0.00141&~0.00070&0.00604&~1.84128&~2.16154&0.963&185\\
~2&~1&~0.00009&~0.00005&0.00941&~1.80584&~2.19435&0.954&103\\
~2&~2&-0.00049&-0.00024&0.00946&~1.80373&~2.19530&0.955&100\\
\hline
\hline
\end{longtable}
\end{center}

\begin{center}
\begin{longtable}{llccccccc}
\caption{\em Simulation results for the normal case with the number of simulated i.i.d.\ draws per sample 400, the number of simulated samples 1000. Different scenarios for the three stopping occasions. Values for standard deviation and $\alpha$ for the probit rule are kept fixed: $\sigma=1$, $\alpha=0$. CL: 95\% confidence limit, Cov.Prob.: coverage probability. \label{simul_normal2}}\\
\hline
\hline
   $\mu$&$\beta$&Bias&Relative Bias& MSE&Lower CL& Upper CL&Cov.Prob.& Aver.Size\\
\hline
\hline
\endfirsthead
\multicolumn{9}{c}%
{\tablename\ \thetable\ -- \it{Continued from previous page}} \\
\hline
\hline
 $\mu$&$\beta$&Bias&Relative Bias& MSE&Lower CL& Upper CL&Cov.Prob.& Aver.Size\\
\hline
\hline
\endhead
\hline \multicolumn{9}{r}{\it{Continued on next page}} \\
\endfoot
\hline
\endlastfoot
\multicolumn{9}{c}{3 Stoppings, $m=50,100,150$, K criterion}\\
\hline
-2&&-0.00284&0.00142&0.02007&-2.27878&-1.7269&0.939&50\\
\hline
\multicolumn{9}{c}{3 Stoppings, $m=50,100,150$, probit criterion}\\
\hline
-2&-2&-0.00284&~0.00142&0.02007&-2.27878&-1.72690&0.939&50\\
-2&-1&-0.00278&~0.00139&0.01988&-2.27687&-1.72869&0.940&51\\
-2&~0&~0.00014&-0.00007&0.01324&-2.22014&-1.77959&0.952&118\\
-2&~1&~0.00189&-0.00095&0.00303&-2.10414&-1.89208&0.959&379\\
-2&~2&~0.00040&-0.00020&0.00235&-2.09764&-1.90157&0.955&400\\
\hline
\multicolumn{9}{c}{3 Stoppings, $m=50,100,150$, K criterion}\\
\hline

-1& &-0.00284&~0.00284&0.02007&-1.27878&-0.7269&0.939&50\\
\hline
\multicolumn{9}{c}{3 Stoppings, $m=50,100,150$, probit criterion}\\
\hline
-1&-2&-0.00332&~0.00332&0.01972&-1.27710&-0.72955&0.939&51\\
-1&-1&-0.00349&~0.00349&0.01824&-1.26446&-0.74251&0.943&62\\
-1&~0&~0.00014&-0.00014&0.01324&-1.22014&-0.77959&0.952&118\\
-1&~1&~0.00604&-0.00604&0.00658&-1.14064&-0.84728&0.952&276\\
-1&~2&~0.00356&-0.00356&0.00323&-1.10266&-0.89023&0.958&378\\
\hline
\multicolumn{9}{c}{3 Stoppings, $m=50,100,150$, K criterion}\\
\hline
~0& &-0.05706&--&0.01184&-0.26121&~0.14709&0.956&171\\
\hline
\multicolumn{9}{c}{3 Stoppings, $m=50,100,150$, probit criterion}\\
\hline
~0&-2&-0.01417&--&0.01256&-0.23159&~0.20325&0.955&122\\
~0&-1&-0.00827&--&0.01279&-0.22817&~0.21164&0.955&119\\
~0&~0&~0.00014&--&0.01324&-0.22014&~0.22041&0.952&118\\
~0&~1&~0.00587&--&0.01330&-0.21403&~0.22577&0.951&118\\
~0&~2&~0.01370&--&0.01322&-0.20658&~0.23397&0.952&118\\
\hline
\multicolumn{9}{c}{3 Stoppings, $m=50,100,150$, K criterion}\\
\hline
~1&&0.00040&0.00040&0.00235&0.90236&1.09843&0.955&400\\
\hline
\multicolumn{9}{c}{3 Stoppings, $m=50,100,150$, probit criterion}\\
\hline
~1&-2&-0.00321&-0.00321&0.00346&~0.89025&~1.10333&0.954&378\\
~1&-1&-0.00534&-0.00534&0.00659&~0.84740&~1.14192&0.960&275\\
~1&~0&~0.00014&~0.00014&0.01324&~0.77986&~1.22041&0.952&118\\
~1&~1&~0.00148&~0.00148&0.01783&~0.73998&~1.26298&0.946&62\\
~1&~2&-0.00085&-0.00085&0.01955&~0.72588&~1.27242&0.940&52\\
\hline
\multicolumn{9}{c}{3 Stoppings, $m=50,100,150$, K criterion}\\
\hline
~2& &~0.00040&~0.00020&0.00235&~1.90236&~2.09843&0.955&400\\
\hline
\multicolumn{9}{c}{3 Stoppings, $m=50,100,150$, probit criterion}\\
\hline
~2&-2&~0.00040&~0.00020&0.00235&~1.90236&~2.09843&0.955&400\\
~2&-1&-0.00106&-0.00053&0.00297&~1.89342&~2.10445&0.956&380\\
~2&~0&~0.00014&~0.00007&0.01324&~1.77986&~2.22041&0.952&118\\
~2&~1&-0.00175&-0.00087&0.01979&~1.72455&~2.27195&0.938&51\\
~2&~2&-0.00264&-0.00132&0.02002&~1.72146&~2.27327&0.939&50\\
\hline
\multicolumn{9}{c}{3 Stoppings, $m=25,50,75$, K criterion}\\
\hline
-2& &-0.00254&~0.00127&0.04125&-2.39018&-1.61489&0.938&25\\
\hline
\multicolumn{9}{c}{3 Stoppings, $m=25,50,75$, probit criterion}\\
\hline
-2&-2&-0.00254&~0.00127&0.04125&-2.39018&-1.61489&0.938&25\\
-2&-1&-0.00441&~0.00221&0.04033&-2.38904&-1.61979&0.938&26\\
-2&~0&-0.00207&~0.00104&0.02729&-2.30824&-1.69591&0.942&85\\
-2&~1&~0.00400&-0.00200&0.00377&-2.10888&-1.88312&0.958&374\\
-2&~2&~0.00040&-0.00020&0.00235&-2.09764&-1.90157&0.955&400\\
\hline
\multicolumn{9}{c}{3 Stoppings, $m=25,50,75$, K criterion}\\
\hline
-1& &-0.00254&~0.00254&0.04125&-1.39018&-0.61489&0.938&25\\
\hline
\multicolumn{9}{c}{3 Stoppings, $m=25,50,75$, probit criterion}\\
\hline
-1&-2&-0.00691&~0.00691&0.03961&-1.39115&-0.62266&0.941&26\\
-1&-1&-0.00679&~0.00679&0.03667&-1.37363&-0.63995&0.942&33\\
-1&~0&-0.00207&~0.00207&0.02729&-1.30824&-0.69591&0.942&85\\
-1&~1&~0.01088&-0.01088&0.01202&-1.17330&-0.80495&0.956&254\\
-1&~2&~0.00757&-0.00757&0.00440&-1.10696&-0.87791&0.956&372\\
\hline
\multicolumn{9}{c}{3 Stoppings, $m=25,50,75$, K criterion}\\
\hline
~0& &-0.08579&--&0.02305&-0.36318&~0.19161&0.953&143\\
\hline
\multicolumn{9}{c}{3 Stoppings, $m=25,50,75$, probit criterion}\\
\hline
~0&-2&-0.02926&--&0.02560&-0.33039&~0.27187&0.948&90\\
~0&-1&-0.01517&--&0.02600&-0.31956&~0.28922&0.947&86\\
~0&~0&-0.00207&--&0.02729&-0.30824&~0.30409&0.942&85\\
~0&~1&~0.00996&--&0.02800&-0.29806&~0.31799&0.946&82\\
~0&~2&~0.02320&--&0.02695&-0.28038&~0.32679&0.950&87\\
\hline
\multicolumn{9}{c}{3 Stoppings, $m=25,50,75$, K criterion}\\
\hline
~1& &~0.00040&~0.00040&0.00235&~0.90236&~1.09843&0.955&400\\
\hline
\multicolumn{9}{c}{3 Stoppings, $m=25,50,75$, probit criterion}\\
\hline
~1&-2&-0.01063&-0.01063&0.00645&~0.87354&~1.10521&0.950&371\\
~1&-1&-0.01469&-0.01469&0.01315&~0.80079&~1.16982&0.953&257\\
~1&~0&-0.00207&-0.00207&0.02729&~0.69176&~1.30409&0.942&85\\
~1&~1&~0.00496&~0.00496&0.03672&~0.63731&~1.37261&0.942&33\\
~1&~2&-0.00027&-0.00027&0.03978&~0.61566&~1.38380&0.940&26\\
\hline
\multicolumn{9}{c}{3 Stoppings, $m=25,50,75$, K criterion}\\
\hline
~2& &~0.00040&~0.00020&0.00235&~1.90236&~2.09843&0.955&400\\
\hline
\multicolumn{9}{c}{3 Stoppings, $m=25,50,75$, probit criterion}\\
\hline
~2&-2&~0.00040&~0.00020&0.00235&~1.90236&~2.09843&0.955&400\\
~2&-1&-0.00356&-0.00178&0.00457&~1.88412&~2.10876&0.956&375\\
~2&~0&-0.00207&-0.00104&0.02729&~1.69176&~2.30409&0.942&85\\
~2&~1&-0.00126&-0.00063&0.04022&~1.61407&~2.38341&0.940&26\\
~2&~2&-0.00254&-0.00127&0.04125&~1.60982&~2.38511&0.938&25\\
\hline
\multicolumn{9}{c}{3 Stoppings, $m=10,20,30$, K criterion}\\
\hline
-2& &~0.00158&-0.00079&0.10434&-2.5983&-1.39855&0.911&10\\
\hline
\multicolumn{9}{c}{3 Stoppings, $m=10,20,30$, probit criterion}\\
\hline
-2&-2&~0.00094&-0.00047&0.10323&-2.59878&-1.39934&0.911&10\\
-2&-1&-0.00185&~0.00093&0.10176&-2.59677&-1.40694&0.911&10\\
-2&~0&~0.00189&-0.00095&0.06994&-2.46607&-1.53014&0.929&65\\
-2&~1&~0.01088&-0.00544&0.00774&-2.11596&-1.86227&0.954&371\\
-2&~2&~0.00040&-0.00020&0.00235&-2.09764&-1.90157&0.955&400\\
\hline
\multicolumn{9}{c}{3 Stoppings, $m=10,20,30$, K criterion}\\
\hline
-1& &~0.00011&-0.00011&0.10225&-1.59953&-0.40025&0.912&10\\
\hline
\multicolumn{9}{c}{3 Stoppings, $m=10,20,30$, probit criterion}\\
\hline
-1&-2&-0.00858&~0.00858&0.09600&-1.59957&-0.41759&0.914&11\\
-1&-1&-0.01666&~0.01666&0.09109&-1.58457&-0.44875&0.918&15\\
-1&~0&~0.00189&-0.00189&0.06994&-1.46607&-0.53014&0.929&65\\
-1&~1&~0.03379&-0.03379&0.03176&-1.22465&-0.70777&0.947&239\\
-1&~2&~0.02406&-0.02406&0.01362&-1.11470&-0.83719&0.950&360\\
\hline
\multicolumn{9}{c}{3 Stoppings, $m=10,20,30$, K criterion}\\
\hline
~0& &-0.14051&--&0.05513&-0.54643&~0.26542&0.936&136\\
\hline
\multicolumn{9}{c}{3 Stoppings, $m=10,20,30$, probit criterion}\\
\hline
~0&-2&-0.06390&--&0.06283&-0.51695&~0.38916&0.935&78\\
~0&-1&-0.03391&--&0.06667&-0.49646&~0.42864&0.934&69\\
~0&~0&~0.00189&--&0.06994&-0.46607&~0.46986&0.929&65\\
~0&~1&~0.03043&--&0.07240&-0.43543&~0.49629&0.924&66\\
~0&~2&~0.06657&--&0.06925&-0.38770&~0.52085&0.934&78\\
\hline
\multicolumn{9}{c}{3 Stoppings, $m=10,20,30$, K criterion}\\
\hline
~1& &-0.00057&-0.00057&0.00337&~0.90082&~1.09804&0.954&400\\
\hline
\multicolumn{9}{c}{3 Stoppings, $m=10,20,30$, probit criterion}\\
\hline
~1&-2&-0.03090&-0.03090&0.01923&~0.82879&~1.10941&0.938&359\\
~1&-1&-0.03707&-0.03707&0.03385&~0.71194&~1.21392&0.934&247\\
~1&~0&~0.00189&~0.00189&0.06994&~0.53393&~1.46986&0.929&65\\
~1&~1&~0.01775&~0.01775&0.09253&~0.45157&~1.58392&0.922&15\\
~1&~2&~0.01001&~0.01001&0.09775&~0.41860&~1.60143&0.919&11\\
\hline
\multicolumn{9}{c}{3 Stoppings, $m=10,20,30$, K criterion}\\
\hline
~2& &0.00040&0.00020&0.00235&1.90236&2.09843&0.955&400\\
\hline
\multicolumn{9}{c}{3 Stoppings, $m=10,20,30$, probit criterion}\\
\hline
~2&-2&~0.00040&~0.00020&0.00235&~1.90236&~2.09843&0.955&400\\
~2&-1&-0.01437&-0.00719&0.01043&~1.85870&~2.11256&0.951&371\\
~2&~0&~0.00189&~0.00095&0.06994&~1.53393&~2.46986&0.929&65\\
~2&~1&~0.00666&~0.00333&0.09977&~1.41215&~2.60117&0.916&10\\
~2&~2&~0.00179&~0.00089&0.10434&~1.40209&~2.60148&0.911&10\\
\hline
\multicolumn{9}{c}{3 Stoppings, $m=5,10,15$, K criterion}\\
\hline
-2& &~0.00499&-0.0025&0.20374&-2.8349&-1.15511&0.884&5\\
\hline
\multicolumn{9}{c}{3 Stoppings, $m=5,10,15$, probit criterion}\\
\hline
-2&-2&~0.00422&-0.00211&0.20275&-2.83500&-1.15655&0.884&5\\
-2&-1&-0.00218&~0.00109&0.19784&-2.83348&-1.17088&0.890&5\\
-2&~0&~0.00014&-0.00007&0.13493&-2.65100&-1.34873&0.903&58\\
-2&~1&~0.01727&-0.00863&0.01399&-2.12791&-1.83755&0.952&369\\
-2&~2&~0.00137&-0.00068&0.00338&-2.09703&-1.90023&0.954&400\\
\hline
\multicolumn{9}{c}{3 Stoppings, $m=5,10,15$, K criterion}\\
\hline
-1& &-0.00094&~0.00094&0.19386&-1.83806&-0.16382&0.884&5\\
\hline
\multicolumn{9}{c}{3 Stoppings, $m=5,10,15$, probit criterion}\\
\hline
-1&-2&-0.01799&~0.01799&0.17915&-1.84115&-0.19482&0.896&6\\
-1&-1&-0.03255&~0.03255&0.17104&-1.81522&-0.24988&0.896&10\\
-1&~0&~0.00014&-0.00014&0.13493&-1.65100&-0.34873&0.903&58\\
-1&~1&~0.07150&-0.07150&0.06649&-1.26519&-0.59181&0.921&235\\
-1&~2&~0.05333&-0.05333&0.03940&-1.12594&-0.76739&0.936&348\\
\hline
\multicolumn{9}{c}{3 Stoppings, $m=5,10,15$, K criterion}\\
\hline
~0& &-0.19706&--&0.10483&-0.75579&~0.36167&0.928&133\\
\hline
\multicolumn{9}{c}{3 Stoppings, $m=5,10,15$, probit criterion}\\
\hline
~0&-2&-0.11859&--&0.12012&-0.74316&~0.50599&0.915&76\\
~0&-1&-0.06873&--&0.12762&-0.71105&~0.57359&0.912&64\\
~0&~0&~0.00014&--&0.13493&-0.65100&~0.65127&0.903&58\\
~0&~1&~0.06384&--&0.13592&-0.57304&~0.70072&0.914&65\\
~0&~2&~0.11084&--&0.12544&-0.50580&~0.72748&0.925&78\\
\hline
\multicolumn{9}{c}{3 Stoppings, $m=5,10,15$, K criterion}\\
\hline
~1& &-0.01086&-0.01086&0.01554&~0.88354&~1.09475&0.946&396\\
\hline
\multicolumn{9}{c}{3 Stoppings, $m=5,10,15$, probit criterion}\\
\hline
~1&-2&-0.06428&-0.06428&0.04773&~0.75278&~1.11865&0.929&346\\
~1&-1&-0.07541&-0.07541&0.07071&~0.58046&~1.26873&0.926&234\\
~1&~0&~0.00014&~0.00014&0.13493&~0.34900&~1.65127&0.903&58\\
~1&~1&~0.02523&~0.02523&0.17868&~0.23621&~1.81426&0.892&9\\
~1&~2&~0.02334&~0.02334&0.18438&~0.20168&~1.84500&0.888&6\\
\hline
\multicolumn{9}{c}{3 Stoppings, $m=5,10,15$, K criterion}\\
\hline
~2& &~0.00040&~0.00020&0.00235&~1.90236&~2.09843&0.955&400\\
\hline
\multicolumn{9}{c}{3 Stoppings, $m=5,10,15$, probit criterion}\\
\hline
~2&-2&-0.00037&-0.00018&0.00302&~1.90071&~2.09855&0.955&400\\
~2&-1&-0.02672&-0.01336&0.02218&~1.82580&~2.12076&0.945&367\\
~2&~0&~0.00014&~0.00007&0.13493&~1.34900&~2.65127&0.903&58\\
~2&~1&~0.01010&~0.00505&0.19822&~1.17678&~2.84343&0.887&5\\
~2&~2&~0.00579&~0.00290&0.20214&~1.16567&~2.84592&0.885&5\\
\hline
\multicolumn{9}{c}{3 Stoppings, $m=2,4,6$, K criterion}\\
\hline
-2& &-0.00268&~0.00134&0.51563&-3.10141&-0.90394&0.687&2\\
\hline
\multicolumn{9}{c}{3 Stoppings, $m=2,4,6$, probit criterion}\\
\hline
-2&-2&-0.01054&~0.00527&0.49998&-3.11239&-0.90870&0.694&2\\
-2&-1&-0.02893&~0.01446&0.48272&-3.12770&-0.93016&0.704&2\\
-2&~0&~0.00501&-0.00250&0.36711&-2.88776&-1.10223&0.778&54\\
-2&~1&~0.06251&-0.03126&0.06741&-2.12864&-1.74633&0.932&360\\
-2&~2&~0.01476&-0.00738&0.02292&-2.09431&-1.87616&0.949&396\\
\hline
\multicolumn{9}{c}{3 Stoppings, $m=2,4,6$, K criterion}\\
\hline
-1& &-0.0634&~0.0634&0.40771&-2.15826&~0.03146&0.730&3\\
\hline
\multicolumn{9}{c}{3 Stoppings, $m=2,4,6$, probit criterion}\\
\hline
-1&-2&-0.07022&~0.07022&0.42042&-2.15092&~0.01048&0.726&3\\
-1&-1&-0.06700&~0.06700&0.42843&-2.11614&-0.01786&0.733&8\\
-1&~0&~0.00501&-0.00501&0.36711&-1.88776&-0.10223&0.778&54\\
-1&~1&~0.17484&-0.17484&0.19682&-1.32545&-0.32487&0.869&222\\
-1&~2&~0.16560&-0.16560&0.17259&-1.13567&-0.53313&0.883&314\\
\hline
\multicolumn{9}{c}{3 Stoppings, $m=2,4,6$, K criterion}\\
\hline
~0& &-0.32619&--&0.27966&-1.07851&~0.42612&0.831&122\\
\hline
\multicolumn{9}{c}{3 Stoppings, $m=2,4,6$, probit criterion}\\
\hline
~0&-2&-0.24247&--&0.30000&-1.04447&~0.55953&0.817&94\\
~0&-1&-0.15688&--&0.32610&-0.99637&~0.68262&0.799&72\\
~0&~0&~0.00501&--&0.36711&-0.88776&~0.89777&0.778&54\\
~0&~1&~0.17018&--&0.30724&-0.68867&~1.02903&0.809&73\\
~0&~2&~0.25107&--&0.28894&-0.57418&~1.07632&0.824&92\\
\hline
\multicolumn{9}{c}{3 Stoppings, $m=2,4,6$, K criterion}\\
\hline
~1& &-0.11654&-0.11654&0.17114&~0.69811&~1.0688&0.902&365\\
\hline
\multicolumn{9}{c}{3 Stoppings, $m=2,4,6$, probit criterion}\\
\hline
~1&-2&-0.17063&-0.17063&0.19322&~0.53865&~1.12009&0.882&313\\
~1&-1&-0.16427&-0.16427&0.21376&~0.35156&~1.31990&0.878&225\\
~1&~0&~0.00501&~0.00501&0.36711&~0.11224&~1.89777&0.778&54\\
~1&~1&~0.07527&~0.07527&0.41361&~0.02334&~2.12721&0.734&10\\
~1&~2&~0.08223&~0.08223&0.38976&~0.00611&~2.15834&0.732&6\\
\hline
\multicolumn{9}{c}{3 Stoppings, $m=2,4,6$, K criterion}\\
\hline
~2& &-0.01081&-0.0054&0.0279&~1.88572&~2.09267&0.951&398\\
\hline
\multicolumn{9}{c}{3 Stoppings, $m=2,4,6$, probit criterion}\\
\hline
~2&-2&-0.02488&-0.01244&0.04915&~1.86658&~2.08366&0.943&394\\
~2&-1&-0.08119&-0.04059&0.09972&~1.72882&~2.10880&0.916&355\\
~2&~0&~0.00501&~0.00250&0.36711&~1.11224&~2.89777&0.778&54\\
~2&~1&~0.02065&~0.01032&0.47933&~0.92699&~3.11431&0.697&2\\
~2&~2&~0.00802&~0.00401&0.49502&~0.90976&~3.10627&0.691&2\\
\hline
\hline
\end{longtable}
\end{center}

\begin{center}
\begin{longtable}{llccccccc}
\caption{\em Simulation results for the Bernoulli case with the number of simulated i.i.d.\ draws per sample 400, the number of simulated samples 1000. Value of $\alpha$ for the probit rule is kept fixed: $\alpha=0$. CL: 95\% confidence limit, Cov.Prob.: coverage probability.\label{simul_bern1}}\\
\hline
\hline
   $\pi$&$\beta$&Bias&Relative Bias& MSE&Lower CL& Upper CL&Cov.Prob.& Aver.Size\\
\hline
\hline
\endfirsthead
\multicolumn{9}{c}%
{\tablename\ \thetable\ -- \it{Continued from previous page}} \\
\hline
\hline
 $\pi$&$\beta$&Bias&Relative Bias& MSE&Lower CL& Upper CL&Cov.Prob.& Aver.Size\\
\hline
\hline
\endhead
\hline \multicolumn{9}{r}{\it{Continued on next page}} \\
\endfoot
\hline
\endlastfoot
\multicolumn{9}{c}{No Stopping} \\
\hline
0.001& &~0.00004&~0.04000&0.00000&-0.00076&0.00284&0.335&400\\
\hline
\multicolumn{9}{c}{1 Stopping, $m=200$, K criterion}\\
\hline
0.001& &~0.00004&~0.04000&0.00000&-0.00076&0.00284&0.335&400\\
\hline
\multicolumn{9}{c}{1 Stopping, $m=200$, probit criterion}\\
\hline
0.001&-2&~0.00002&~0.02500&0.00000&-0.00079&0.00284&0.257&296\\
0.001&-1&~0.00003&~0.02750&0.00000&-0.00079&0.00285&0.257&296\\
0.001&~0&~0.00003&~0.02750&0.00000&-0.00079&0.00285&0.257&296\\
0.001&~1&~0.00004&~0.03500&0.00000&-0.00080&0.00287&0.257&296\\
0.001&~2&~0.00004&~0.03500&0.00000&-0.00080&0.00287&0.257&296\\
\hline
\multicolumn{9}{c}{3 Stoppings, $m=100,200,300$, K criterion}\\
\hline
0.001&	&~0.00004&~0.04000&0.00000&-0.00076&0.00284&0.335&400\\
\hline
\multicolumn{9}{c}{3 Stoppings, $m=100,200,300$, probit criterion}\\
\hline
0.001&-2&-0.00001&-0.01333&0.00001&-0.00085&0.00283&0.169&185\\
0.001&-1&-0.00001&-0.01167&0.00001&-0.00086&0.00283&0.169&185\\
0.001&~0&-0.00001&-0.01167&0.00001&-0.00086&0.00283&0.169&185\\
0.001&~1&-0.00001&-0.00667&0.00001&-0.00086&0.00285&0.169&185\\
0.001&~2&-0.00001&-0.00667&0.00001&-0.00086&0.00285&0.169&185\\
\hline
\multicolumn{9}{c}{No Stopping} \\
\hline
0.010& &-0.00009&-0.00925&0.00003&~0.00060&0.01922&0.906&400\\
\hline
\multicolumn{9}{c}{1 Stopping, $m=200$, K criterion}\\
\hline
0.010& &-0.00009&-0.00925&0.00003&~0.00060&0.01922&0.906&400\\
\hline
\multicolumn{9}{c}{1 Stopping, $m=200$, proit criterion}\\
\hline
0.010&-2&-0.00016&-0.01550&0.00004&-0.00087&0.02056&0.875&298\\
0.010&-1&-0.00014&-0.01375&0.00004&-0.00088&0.02060&0.875&297\\
0.010&~0&-0.00012&-0.01200&0.00004&-0.00089&0.02065&0.876&296\\
0.010&~1&-0.00008&-0.00800&0.00004&-0.00090&0.02074&0.876&295\\
0.010&~2&-0.00007&-0.00725&0.00004&-0.00092&0.02077&0.876&294\\
\hline
\multicolumn{9}{c}{3 Stoppings, $m=100,200,300$, K criterion}\\
\hline
0.010& &-0.00009&-0.00925&0.00003&0.00060&0.01922&0.906&400\\
\hline
\multicolumn{9}{c}{3 Stoppings, $m=100,200,300$, probit criterion}\\
\hline
0.010&-2&-0.00019&-0.01900&0.00007&-0.00308&0.02270&0.733&186\\
0.010&-1&-0.00018&-0.01825&0.00007&-0.00311&0.02275&0.732&186\\
0.010&~0&-0.00011&-0.01142&0.00007&-0.00311&0.02288&0.731&185\\
0.010&~1&-0.00011&-0.01075&0.00007&-0.00316&0.02295&0.731&183\\
0.010&~2&-0.00008&-0.00783&0.00007&-0.00319&0.02303&0.731&183\\
\hline
\multicolumn{9}{c}{No Stopping} \\
\hline
0.100& &-0.00038&-0.00380&0.00022&~0.07037&0.12887&0.950&400\\
\hline
\multicolumn{9}{c}{1 Stopping, $m=200$, K criterion}\\
\hline
0.100& &-0.00038&-0.00380&.00022&~0.07037&0.12887&0.950&400\\
\hline
\multicolumn{9}{c}{1 Stopping, $m=200$, probit criterion}\\
\hline
0.100&-2&-0.00035&-0.00350&0.00030&~0.06529&0.13401&0.944&315\\
0.100&-1&-0.00039&-0.00385&0.00032&~0.06474&0.13449&0.941&306\\
0.100&~0&-0.00027&-0.00268&0.00032&~0.06424&0.13522&0.941&296\\
0.100&~1&-0.00038&-0.00378&0.00033&~0.06368&0.13556&0.939&288\\
0.100&~2&-0.00050&-0.00500&0.00034&~0.06327&0.13573&0.938&283\\
\hline
\multicolumn{9}{c}{3 Stoppings, $m=100,200,300$, K criterion}\\
\hline
0.100& &-0.00038&-0.00380&0.00022&~0.07037&0.12887&0.950&400\\
\hline
\multicolumn{9}{c}{3 Stoppings, $m=100,200,300$, probit criterion}\\
\hline
0.10&-2&-0.00079&-0.00793&0.00060&~0.05446&0.14395&0.935&210\\
0.10&-1&-0.00034&-0.00336&0.00064&~0.05344&0.14589&0.939&196\\
0.10&~0&-0.00005&-0.00053&0.00066&~0.05241&0.14749&0.939&185\\
0.10&~1&-0.00003&-0.00029&0.00068&~0.05130&0.14864&0.939&174\\
0.10&~2&~0.00034&~0.00341&0.00071&~0.05071&0.14997&0.935&165\\
\hline
\multicolumn{9}{c}{No Stopping} \\
\hline
0.30& &~0.00058&~0.00192&0.00054&~0.25571&0.34544&0.946&400\\
\hline
\multicolumn{9}{c}{1 Stopping, $m=200$, K criterion}\\
\hline
0.30& &~0.00058&~0.00192&0.00054&~0.25571&0.34544&0.946&400\\
\hline
\multicolumn{9}{c}{1 Stopping, $m=200$, probit criterion}\\
\hline
0.30&-2&~0.00066&~0.00221&0.00069&~0.25089&0.35043&0.945&347\\
0.30&-1&~0.00084&~0.00278&0.00076&~0.24888&0.35279&0.945&323\\
0.30&~0&~0.00069&~0.00230&0.00080&~0.24620&0.35518&0.947&296\\
0.30&~1&~0.00038&~0.00128&0.00083&~0.24408&0.35668&0.947&276\\
0.30&~2&~0.00088&~0.00294&0.00090&~0.24250&0.35927&0.943&254\\
\hline
\multicolumn{9}{c}{3 Stoppings, $m=100,200,300$, K criterion}\\
\hline
0.30& &0.00058&0.00192&0.00054&0.25571&0.34544&0.946&400\\
\hline
\multicolumn{9}{c}{3 Stoppings, $m=100,200,300$, probit criterion}\\
\hline
0.30&-2&-0.00028&-0.00092&0.00117&~0.23847&0.36098&0.935&265\\
0.30&-1&-0.00008&-0.00027&0.00140&~0.23261&0.36722&0.944&222\\
0.30&~0&~0.00068&~0.00227&0.00161&~0.22751&0.37385&0.939&185\\
0.30&~1&~0.00108&~0.00359&0.00177&~0.22372&0.37844&0.939&158\\
0.30&~2&~0.00201&~0.00668&0.00193&~0.22084&0.38317&0.941&138\\
\hline
\multicolumn{9}{c}{No Stopping} \\
\hline
0.50& &~0.00089&~0.00177&0.00063&~0.45195&0.54983&0.943&400\\
\hline
\multicolumn{9}{c}{1 Stopping, $m=200$, K criterion}\\
\hline
0.50& &~0.00089&~0.00177&0.00063&~0.45195&0.54983&0.943&400\\
\hline
\multicolumn{9}{c}{1 Stopping, $m=200$, probit criterion}\\
\hline
0.50&-2&~0.00056&~0.00112&0.00071&~0.44841&0.55271&0.944&368\\
0.50&-1&~0.00069&~0.00137&0.00078&~0.44571&0.55567&0.948&340\\
0.50&~0&~0.00055&~0.00110&0.00088&~0.44111&0.56000&0.954&296\\
0.50&~1&-0.00002&-0.00003&0.00098&~0.43719&0.56278&0.955&263\\
0.50&~2&~0.00005&~0.00009&0.00109&~0.43408&0.56601&0.952&231\\
\hline
\multicolumn{9}{c}{3 Stoppings, $m=100,200,300$, K criterion}\\
\hline
0.50& &0.00089&0.00177&0.00063&0.45195&0.54983&0.943&400\\
\hline
\multicolumn{9}{c}{3 Stoppings, $m=100,200,300$, probit criterion}\\
\hline
0.50&-2&-0.00144&-0.00288&0.00109&~0.43808&0.55904&0.944&312\\
0.50&-1&-0.00132&-0.00264&0.00141&~0.42961&0.56775&0.944&251\\
0.50&~0&-0.00031&-0.00063&0.00183&~0.41983&0.57954&0.943&185\\
0.50&~1&~0.00011&~0.00023&0.00217&~0.41298&0.58724&0.938&144\\
0.50&~2&-0.00031&-0.00062&0.00240&~0.40715&0.59223&0.938&120\\
\hline
\multicolumn{9}{c}{No Stopping} \\
\hline
0.70&-0.00057&-0.00082&0.00054&~0.65456&0.74429&0.946&400\\
\hline
\multicolumn{9}{c}{1 Stopping, $m=200$, K criterion}\\
\hline
0.70& &-0.00057&-0.00082&0.00054&~0.65456&0.74429&0.946&400\\
\hline
\multicolumn{9}{c}{1 Stopping, $m=200$, probit criterion}\\
\hline
0.70&-2&-0.00095&-0.00136&0.00059&~0.65269&0.74541&0.946&384\\
0.70&-1&-0.00116&-0.00165&0.00068&~0.64957&0.74812&0.947&353\\
0.70&~0&-0.00069&-0.00099&0.00080&~0.64482&0.75380&0.947&296\\
0.70&~1&-0.00033&-0.00047&0.00093&~0.64074&0.75859&0.942&248\\
0.70&~2&-0.00017&-0.00025&0.00102&~0.63822&0.76143&0.941&219\\
\hline
\multicolumn{9}{c}{3 Stoppings, $m=100,200,300$, K criterion}\\
\hline
0.70& &-0.00057&-0.00082&0.00054&~0.65456&0.74429&0.946&400\\
\hline
\multicolumn{9}{c}{3 Stoppings, $m=100,200,300$, probit criterion}\\
\hline
0.70&-2&-0.00135&-0.00192&0.00079&~0.64838&0.74893&0.945&355\\
0.70&-1&-0.00168&-0.00240&0.00110&~0.63835&0.75828&0.940&277\\
0.70&~0&-0.00068&-0.00097&0.00161&~0.62615&0.77249&0.939&185\\
0.70&~1&-0.00061&-0.00087&0.00193&~0.61733&0.78145&0.944&133\\
0.70&~2&-0.00070&-0.00100&0.00217&~0.61249&0.78612&0.938&110\\
\hline
\multicolumn{9}{c}{No Stopping} \\
\hline
0.90& &~0.00038&~0.00042&0.00022&~0.87113&0.92963&0.950&400\\
\hline
\multicolumn{9}{c}{1 Stopping, $m=200$, K criterion}\\
\hline
0.90& &~0.00038&~0.00042&0.00022&~0.87113&0.92963&0.950&400\\
\hline
\multicolumn{9}{c}{1 Stopping, $m=200$, probit criterion}\\
\hline
0.90&-2&~0.00027&~0.00030&0.00022&~0.87067&0.92988&0.950&394\\
0.90&-1&~0.00013&~0.00014&0.00025&~0.86863&0.93163&0.949&363\\
0.90&~0&~0.00027&~0.00030&0.00032&~0.86478&0.93576&0.941&296\\
0.90&~1&~0.00068&~0.00076&0.00041&~0.86176&0.93960&0.929&237\\
0.90&~2&~0.00039&~0.00043&0.00046&~0.85964&0.94114&0.926&207\\
\hline
\multicolumn{9}{c}{3 Stoppings, $m=100,200,300$, K criterion}\\
\hline
0.90& &~0.00038&~0.00042&0.00022&~0.87113&0.92963&0.950&400\\
\hline
\multicolumn{9}{c}{3 Stoppings, $m=100,200,300$, probit criterion}\\
\hline
0.90&-2&~0.00011&~0.00013&0.00025&~0.86929&0.93094&0.949&379\\
0.90&-1&-0.00009&-0.00010&0.00038&~0.86284&0.93698&0.947&301\\
0.90&~0&~0.00005&~0.00006&0.00066&~0.85251&0.94759&0.939&185\\
0.90&~1&~0.00037&~0.00041&0.00083&~0.84600&0.95473&0.933&124\\
0.90&~2&~0.00049&~0.00054&0.00090&~0.84333&0.95765&0.932&104\\
\hline
\hline
\end{longtable}
\end{center}

\newpage
\begin{center}
\begin{longtable}{llccccccc}
\caption{\em Simulation results for the Bernoulli case with the number of simulated i.i.d.\ draws per sample 400, the number of simulated samples 1000. Different scenarios for the three stopping occasions. Value of $\alpha$ for the probit rule is kept fixed: $\alpha=0$. CL: 95\% confidence limit, Cov.Prob.: coverage probability.\label{simul_bern2}}\\
\hline
\hline
   $\pi$&$\beta$&Bias&Relative Bias& MSE&Lower CL& Upper CL&Cov.Prob.& Aver.Size\\
\hline
\hline
\endfirsthead
\multicolumn{9}{c}%
{\tablename\ \thetable\ -- \it{Continued from previous page}} \\
\hline
\hline
 $\pi$&$\beta$&Bias&Relative Bias& MSE&Lower CL& Upper CL&Cov.Prob.& Aver.Size\\
\hline
\hline
\endhead
\hline \multicolumn{9}{r}{\it{Continued on next page}} \\
\endfoot
\hline
\endlastfoot
\multicolumn{9}{c}{3 Stoppings, $m=50,100,150$, K criterion}\\
\hline
0.001& &~0.00004&~0.04000&0.00000&-0.00076&0.00284&0.335&400\\
\hline
\multicolumn{9}{c}{3 Stoppings, $m=50,100,150$, probit criterion}\\
\hline
0.001&-2&~0.00002&~0.01833&0.00001&-0.00091&0.00295&0.109&118\\
0.001&-1&~0.00002&~0.01833&0.00001&-0.00091&0.00295&0.109&118\\
0.001&~0&~0.00002&~0.01833&0.00001&-0.00091&0.00295&0.109&118\\
0.001&~1&~0.00002&~0.01833&0.00001&-0.00091&0.00295&0.109&118\\
0.001&~2&~0.00002&~0.01833&0.00001&-0.00091&0.00295&0.109&118\\
\hline
\multicolumn{9}{c}{3 Stoppings, $m=50,100,150$, K criterion}\\
\hline
0.010& &-0.00009&-0.00925&0.00003&~0.00060&0.01922&0.906&400\\
\hline
\multicolumn{9}{c}{3 Stoppings, $m=50,100,150$, probit criterion}\\
\hline
0.010&-2&-0.00010&-0.00958&0.00013&-0.00503&0.02484&0.557&119\\
0.010&-1&-0.00004&-0.00408&0.00013&-0.00508&0.02500&0.557&118\\
0.010&~0&-0.00003&-0.00342&0.00013&-0.00515&0.02508&0.557&118\\
0.010&~1&~0.00004&~0.00367&0.00013&-0.00524&0.02531&0.557&117\\
0.010&~2&~0.00013&~0.01267&0.00014&-0.00530&0.02556&0.557&116\\
\hline
\multicolumn{9}{c}{3 Stoppings, $m=50,100,150$, K criterion}\\
\hline
0.100& &-0.00038&-0.00380&0.00022&~0.07037&0.12887&0.950&400\\
\hline
\multicolumn{9}{c}{3 Stoppings, $m=50,100,150$, probit criterion}\\
\hline
0.100&-2&-0.00091&-0.00906&0.00111&~0.03865&0.15953&0.911&142\\
0.100&-1&~0.00002&~0.00023&0.00120&~0.03713&0.16292&0.905&128\\
0.100&~0&~0.00042&~0.00417&0.00127&~0.03538&0.16546&0.902&118\\
0.100&~1&~0.00080&~0.00799&0.00134&~0.03395&0.16765&0.900&108\\
0.100&~2&~0.00132&~0.01317&0.00145&~0.03256&0.17007&0.897&97\\
\hline
\multicolumn{9}{c}{3 Stoppings, $m=50,100,150$, K criterion}\\
\hline
0.300& &~0.00058&~0.00192&0.00054&~0.25571&0.34544&0.946&400\\
\hline
\multicolumn{9}{c}{3 Stoppings, $m=50,100,150$, probit criterion}\\
\hline
0.300&-2&-0.00147&-0.00489&0.00212&~0.21895&0.37811&0.929&209\\
0.300&-1&-0.00129&-0.00429&0.00259&~0.20843&0.38899&0.936&158\\
0.300&~0&~0.00014&~0.00047&0.00306&~0.19961&0.40067&0.933&118\\
0.300&~1&~0.00114&~0.00380&0.00341&~0.19329&0.40899&0.933&90\\
0.300&~2&~0.00099&~0.00331&0.00360&~0.18751&0.41447&0.934&75\\
\hline
\multicolumn{9}{c}{3 Stoppings, $m=50,100,150$, K criterion}\\
\hline
0.500& &~0.00089&~0.00177&0.00063&~0.45195&0.54983&0.943&400\\
\hline
\multicolumn{9}{c}{3 Stoppings, $m=50,100,150$, probit criterion}\\
\hline
0.500&-2&-0.00452&-0.00904&0.00192&~0.42227&0.56869&0.940&276\\
0.500&-1&-0.00447&-0.00893&0.00272&~0.40515&0.58591&0.937&194\\
0.500&~0&-0.00258&-0.00515&0.00353&~0.38753&0.60732&0.933&118\\
0.500&~1&-0.00164&-0.00327&0.00432&~0.37636&0.62037&0.926&79\\
0.500&~2&-0.00220&-0.00441&0.00486&~0.36781&0.62778&0.925&61\\
\hline
\multicolumn{9}{c}{3 Stoppings, $m=50,100,150$, K criterion}\\
\hline
0.700& &-0.00057&-0.00082&0.00054&~0.65456&0.74429&0.946&400\\
\hline
\multicolumn{9}{c}{3 Stoppings, $m=50,100,150$, probit criterion}\\
\hline
0.700&-2&-0.00319&-0.00456&0.00110&~0.64010&0.75352&0.944&334\\
0.700&-1&-0.00308&-0.00441&0.00195&~0.61998&0.77385&0.937&224\\
0.700&~0&-0.00014&-0.00020&0.00306&~0.59933&0.80039&0.933&118\\
0.700&~1&~0.00178&~0.00254&0.00370&~0.58693&0.81662&0.933&70\\
0.700&~2&~0.00148&~0.00211&0.00407&~0.57985&0.82311&0.933&56\\
\hline
\multicolumn{9}{c}{3 Stoppings, $m=50,100,150$, K criterion}\\
\hline
0.900& &0.00038&0.00042&0.00022&0.87113&0.92963&0.950&400\\
\hline
\multicolumn{9}{c}{3 Stoppings, $m=50,100,150$, probit criterion}\\
\hline
0.900&-2&-0.00015&-0.00017&0.00030&~0.86703&0.93266&0.948&369\\
0.900&-1&-0.00076&-0.00084&0.00064&~0.85332&0.94517&0.940&258\\
0.900&~0&-0.00042&-0.00046&0.00127&~0.83454&0.96462&0.902&118\\
0.900&~1&-0.00016&-0.00018&0.00174&~0.82424&0.97544&0.878&64\\
0.900&~2&-0.00066&-0.00073&0.00191&~0.81949&0.97919&0.872&52\\
\hline
\multicolumn{9}{c}{3 Stoppings, $m=25,50,75$, K criterion}\\
\hline
0.001& &0.00004&0.04000&0.00000&-0.00076&0.00284&0.335&400\\
\hline
\multicolumn{9}{c}{3 Stoppings, $m=25,50,75$, probit criterion}\\
\hline
0.001&-2&~0.00006&~0.06083&0.00003&-0.00094&0.00306&0.076&85\\
0.001&-1&~0.00008&~0.07833&0.00003&-0.00096&0.00311&0.076&85\\
0.001&~0&~0.00008&~0.07833&0.00003&-0.00096&0.00311&0.076&85\\
0.001&~1&~0.00008&~0.07833&0.00003&-0.00096&0.00311&0.076&85\\
0.001&~2&~0.00011&~0.11167&0.00003&-0.00099&0.00321&0.076&85\\
\hline
\multicolumn{9}{c}{3 Stoppings, $m=25,50,75$, K criterion}\\
\hline
0.010& &-0.00009&-0.00925&0.00003&~0.00060&0.01922&0.906&400\\
\hline
\multicolumn{9}{c}{3 Stoppings, $m=25,50,75$, probit criterion}\\
\hline
0.010&-2&-0.00036&-0.03617&0.00026&-0.00592&0.02520&0.379&86\\
0.010&-1&-0.00022&-0.02200&0.00027&-0.00605&0.02561&0.379&85\\
0.010&~0&-0.00007&-0.00675&0.00028&-0.00615&0.02602&0.379&85\\
0.010&~1&-0.00007&-0.00675&0.00028&-0.00618&0.02604&0.379&85\\
0.010&~2&~0.00005&~0.00483&0.00029&-0.00633&0.02642&0.379&84\\
\hline
\multicolumn{9}{c}{3 Stoppings, $m=25,50,75$, K criterion}\\
\hline
0.100& &-0.00038&-0.00380&0.00022&~0.07037&0.12887&0.950&400\\
\hline
\multicolumn{9}{c}{3 Stoppings, $m=25,50,75$, probit criterion}\\
\hline
0.100&-2&-0.00236&-0.02358&0.00215&~0.01869&0.17660&0.915&108\\
0.100&-1&-0.00098&-0.00976&0.00239&~0.01628&0.18177&0.917&95\\
0.100&~0&~0.00102&~0.01020&0.00258&~0.01410&0.18794&0.918&85\\
0.100&~1&~0.00223&~0.02227&0.00276&~0.01243&0.19203&0.917&77\\
0.100&~2&~0.00322&~0.03220&0.00297&~0.00998&0.19646&0.916&64\\
\hline
\multicolumn{9}{c}{3 Stoppings, $m=25,50,75$, K criterion}\\
\hline
0.300& &~0.00058&~0.00192&0.00054&~0.25571&0.34544&0.946&400\\
\hline
\multicolumn{9}{c}{3 Stoppings, $m=25,50,75$, probit criterion}\\
\hline
0.300&-2&-0.00569&-0.01897&0.00391&~0.19044&0.39817&0.941&181\\
0.300&-1&-0.00287&-0.00957&0.00474&~0.17523&0.41903&0.949&125\\
0.300&~0&~0.00137&~0.00456&0.00582&~0.16280&0.43994&0.947&85\\
0.300&~1&~0.00272&~0.00906&0.00658&~0.15221&0.45322&0.948&56\\
0.300&~2&~0.00534&~0.01781&0.00753&~0.14640&0.46429&0.942&42\\
\hline
\multicolumn{9}{c}{3 Stoppings, $m=25,50,75$, K criterion}\\
\hline
0.500& &0.00089&0.00177&0.00063&0.45195&0.54983&0.943&400\\
\hline
\multicolumn{9}{c}{3 Stoppings, $m=25,50,75$, probit criterion}\\
\hline
0.500&-2&-0.00724&-0.01448&0.00359&~0.40141&0.58412&0.942&255\\
0.500&-1&-0.00596&-0.01193&0.00513&~0.37358&0.61449&0.944&164\\
0.500&~0&-0.00163&-0.00325&0.00676&~0.34676&0.64999&0.957&85\\
0.500&~1&~0.00063&~0.00126&0.00843&~0.33018&0.67107&0.954&46\\
0.500&~2&~0.00054&~0.00107&0.00953&~0.31942&0.68165&0.952&32\\
\hline
\multicolumn{9}{c}{3 Stoppings, $m=25,50,75$, K criterion}\\
\hline
0.700& &-0.00057&-0.00082&0.00054&0.65456&0.74429&0.946&400\\
\hline
\multicolumn{9}{c}{3 Stoppings, $m=25,50,75$, probit criterion}\\
\hline
0.700&-2&-0.00514&-0.00734&0.00175&~0.62967&0.76005&0.946&325\\
0.700&-1&-0.00251&-0.00359&0.00360&~0.59786&0.79711&0.940&200\\
0.700&~0&-0.00137&-0.00195&0.00582&~0.56006&0.83720&0.947&85\\
0.700&~1&~0.00076&~0.00109&0.00744&~0.54079&0.86074&0.942&39\\
0.700&~2&~0.00010&~0.00014&0.00833&~0.52994&0.87026&0.939&28\\
\hline
\multicolumn{9}{c}{3 Stoppings, $m=25,50,75$, K criterion}\\
\hline
0.900& &0.00038&0.00042&0.00022&0.87113&0.92963&0.950&400\\
\hline
\multicolumn{9}{c}{3 Stoppings, $m=25,50,75$, probit criterion}\\
\hline
0.900&-2&-0.00067&-0.00074&0.00041&~0.86389&0.93478&0.948&364\\
0.900&-1&-0.00090&-0.00100&0.00128&~0.84251&0.95570&0.938&239\\
0.900&~0&-0.00102&-0.00113&0.00258&~0.81206&0.98590&0.918&85\\
0.900&~1&-0.00002&-0.00003&0.00346&~0.79898&1.00098&0.905&34\\
0.900&~2&~0.00007&~0.00007&0.00383&~0.79412&1.00601&0.903&26\\
\hline
\multicolumn{9}{c}{3 Stoppings, $m=10,20,30$, K criterion}\\
\hline
0.001& &~0.00004&~0.04000&0.00000&-0.00076&0.00284&0.335&400\\
\hline
\multicolumn{9}{c}{3 Stoppings, $m=10,20,30$, probit criterion}\\
\hline
0.001&-2&~0.00010&~0.09667&0.00007&-0.00091&0.00310&0.056&65\\
0.001&-1&~0.00019&~0.19417&0.00008&-0.00099&0.00338&0.056&65\\
0.001&~0&~0.00022&~0.22500&0.00008&-0.00102&0.00347&0.056&65\\
0.001&~1&~0.00024&~0.24167&0.00008&-0.00104&0.00352&0.056&65\\
0.001&~2&~0.00024&~0.24167&0.00008&-0.00104&0.00352&0.056&65\\
\hline
\multicolumn{9}{c}{3 Stoppings, $m=10,20,30$, K criterion}\\
\hline
0.010& &-0.00009&-0.00925&0.00003&~0.00060&0.01922&0.906&400\\
\hline
\multicolumn{9}{c}{3 Stoppings, $m=10,20,30$, probit criterion}\\
\hline
0.010&-2&-0.00107&-0.10683&0.00056&-0.00634&0.02420&0.234&67\\
0.010&-1&-0.00074&-0.07408&0.00059&-0.00664&0.02516&0.234&66\\
0.010&~0&-0.00053&-0.05308&0.00061&-0.00686&0.02580&0.234&65\\
0.010&~1&-0.00010&-0.00992&0.00067&-0.00714&0.02694&0.234&63\\
0.010&~2&~0.00014&~0.01358&0.00069&-0.00728&0.02755&0.234&63\\
\hline
\multicolumn{9}{c}{3 Stoppings, $m=10,20,30$, K criterion}\\
\hline
0.100& &-0.00038&-0.00380&0.00022&0.07037&0.12887&0.950&400\\
\hline
\multicolumn{9}{c}{3 Stoppings, $m=10,20,30$, probit criterion}\\
\hline
0.100&-2&-0.00644&-0.06438&0.00460&-0.00696&0.19408&0.763&93\\
0.100&-1&-0.00254&-0.02539&0.00521&-0.01238&0.20730&0.768&75\\
0.100&~0&~0.00142&~0.01417&0.00597&-0.01610&0.21894&0.767&65\\
0.100&~1&~0.00542&~0.05422&0.00666&-0.01876&0.22960&0.767&57\\
0.100&~2&~0.00797&~0.07967&0.00723&-0.02260&0.23853&0.769&46\\
\hline
\multicolumn{9}{c}{3 Stoppings, $m=10,20,30$, K criterion}\\
\hline
0.300& &~0.00058&~0.00192&0.00054&~0.25571&0.34544&0.946&400\\
\hline
\multicolumn{9}{c}{3 Stoppings, $m=10,20,30$, probit criterion}\\
\hline
0.300&-2&-0.01617&-0.05390&0.00949&~0.13942&0.42824&0.890&163\\
0.300&-1&-0.00842&-0.02806&0.01231&~0.11491&0.46825&0.887&105\\
0.300&~0&-0.00036&-0.00119&0.01509&~0.09422&0.50506&0.884&65\\
0.300&~1&~0.00506&~0.01687&0.01720&~0.07927&0.53085&0.883&35\\
0.300&~2&~0.00837&~0.02790&0.01885&~0.06856&0.54818&0.878&24\\
\hline
\multicolumn{9}{c}{3 Stoppings, $m=10,20,30$, K criterion}\\
\hline
0.500& &0.00089&0.00177&0.00063&0.45195&0.54983&0.943&400\\
\hline
\multicolumn{9}{c}{3 Stoppings, $m=10,20,30$, probit criterion}\\
\hline
0.500&-2&-0.01862&-0.03725&0.00892&~0.35580&0.60696&0.935&242\\
0.500&-1&-0.00997&-0.01994&0.01279&~0.31347&0.66660&0.932&145\\
0.500&~0&-0.00023&-0.00046&0.01778&~0.27093&0.72861&0.917&65\\
0.500&~1&~0.00439&~0.00878&0.02133&~0.24596&0.76283&0.902&27\\
0.500&~2&~0.00444&~0.00888&0.02342&~0.22892&0.77996&0.894&15\\
\hline
\multicolumn{9}{c}{3 Stoppings, $m=10,20,30$, K criterion}\\
\hline
0.700& &-0.00057&-0.00082&0.00054&0.65456&0.74429&0.946&400\\
\hline
\multicolumn{9}{c}{3 Stoppings, $m=10,20,30$, probit criterion}\\
\hline
0.700&-2&-0.01218&-0.01740&0.00455&~0.60216&0.77347&0.943&311\\
0.700&-1&-0.00841&-0.01201&0.00936&~0.54889&0.83429&0.917&183\\
0.700&~0&~0.00036&~0.00051&0.01509&~0.49494&0.90578&0.884&65\\
0.700&~1&~0.00595&~0.00850&0.01902&~0.46825&0.94365&0.848&19\\
0.700&~2&~0.00223&~0.00319&0.01998&~0.44891&0.95555&0.846&11\\
\hline
\multicolumn{9}{c}{3 Stoppings, $m=10,20,30$, K criterion}\\
\hline
0.900& &0.00038&0.00042&0.00022&0.87113&0.92963&0.950&400\\
\hline
\multicolumn{9}{c}{3 Stoppings, $m=10,20,30$, probit criterion}\\
\hline
0.900&-2&-0.00263&-0.00292&0.00099&~0.85565&0.93909&0.942&357\\
0.900&-1&-0.00269&-0.00299&0.00273&~0.82267&0.97195&0.879&224\\
0.900&~0&-0.00142&-0.00157&0.00597&~0.78106&1.01610&0.767&65\\
0.900&~1&-0.00016&-0.00018&0.00792&~0.76663&1.03305&0.681&16\\
0.900&~2&-0.00187&-0.00207&0.00885&~0.75822&1.03805&0.657&10\\
\hline
\multicolumn{9}{c}{3 Stoppings, $m=5,10,15$, K criterion}\\
\hline
0.001& &0.00004&0.04000&0.00000&-0.00076&0.00284&0.335&400\\
\hline
\multicolumn{9}{c}{3 Stoppings, $m=5,10,15$, probit criterion}\\
\hline
0.001&-2&~0.00036&~0.36083&0.00015&-0.00112&0.00384&0.051&58\\
0.001&-1&~0.00046&~0.46083&0.00018&-0.00118&0.00410&0.051&58\\
0.001&~0&~0.00053&~0.52500&0.00019&-0.00124&0.00429&0.051&58\\
0.001&~1&~0.00056&~0.55833&0.00019&-0.00126&0.00438&0.051&58\\
0.001&~2&~0.00072&~0.72500&0.00023&-0.00138&0.00483&0.051&58\\
\hline
\multicolumn{9}{c}{3 Stoppings, $m=5,10,15$, K criterion}\\
\hline
0.010& &-0.00009&-0.00925&0.00003&0.00060&0.01922&0.906&400\\
\hline
\multicolumn{9}{c}{3 Stoppings, $m=5,10,15$, probit criterion}\\
\hline
0.010&-2&-0.00133&-0.13267&0.00098&-0.00582&0.02317&0.185&61\\
0.010&-1&-0.00040&-0.03975&0.00118&-0.00653&0.02573&0.185&59\\
0.010&~0&~0.00045&~0.04525&0.00134&-0.00722&0.02812&0.185&58\\
0.010&~1&~0.00138&~0.13842&0.00154&-0.00792&0.03069&0.185&57\\
0.010&~2&~0.00201&~0.20108&0.00167&-0.00849&0.03251&0.185&55\\
\hline
\multicolumn{9}{c}{3 Stoppings, $m=5,10,15$, K criterion}\\
\hline
0.100& &-0.00038&-0.00380&0.00022&~0.07037&0.12887&0.950&400\\
\hline
\multicolumn{9}{c}{3 Stoppings, $m=5,10,15$, probit criterion}\\
\hline
0.100&-2&-0.01102&-0.11019&0.00829&-0.01726&0.19522&0.590&85\\
0.100&-1&-0.00372&-0.03716&0.01003&-0.02464&0.21721&0.590&70\\
0.100&~0&~0.00370&~0.03700&0.01172&-0.03124&0.23864&0.591&58\\
0.100&~1&~0.00841&~0.08406&0.01284&-0.03569&0.25250&0.592&50\\
0.100&~2&~0.01413&~0.14126&0.01508&-0.04107&0.26932&0.587&40\\
\hline
\multicolumn{9}{c}{3 Stoppings, $m=5,10,15$, K criterion}\\
\hline
0.300& &0.00058&0.00192&0.00054&0.25571&0.34544&0.946&400\\
\hline
\multicolumn{9}{c}{3 Stoppings, $m=5,10,15$, probit criterion}\\
\hline
0.300&-2&-0.03043&-0.10143&0.01704&~0.09311&0.44603&0.850&154\\
0.300&-1&-0.01489&-0.04962&0.02187&~0.06332&0.50690&0.852&98\\
0.300&~0&~0.00108&~0.00358&0.02818&~0.03722&0.56493&0.845&58\\
0.300&~1&~0.01274&~0.04246&0.03262&~0.01832&0.60715&0.847&30\\
0.300&~2&~0.01986&~0.06621&0.03594&~0.00686&0.63287&0.842&18\\
\hline
\multicolumn{9}{c}{3 Stoppings, $m=5,10,15$, K criterion}\\
\hline
0.500& &~0.00089&~0.00177&0.00063&~0.45195&0.54983&0.943&400\\
\hline
\multicolumn{9}{c}{3 Stoppings, $m=5,10,15$, probit criterion}\\
\hline
0.500&-2&-0.03272&-0.06544&0.01516&~0.30543&0.62913&0.937&236\\
0.500&-1&-0.02032&-0.04063&0.02252&~0.24597&0.71340&0.926&135\\
0.500&~0&-0.00118&-0.00236&0.03350&~0.19676&0.80088&0.922&58\\
0.500&~1&~0.01307&~0.02614&0.03974&~0.17034&0.85580&0.934&19\\
0.500&~2&~0.01661&~0.03321&0.04331&~0.15396&0.87925&0.940&9\\
\hline
\multicolumn{9}{c}{3 Stoppings, $m=5,10,15$, K criterion}\\
\hline
0.700& &-0.00057&-0.00082&0.00054&0.65456&0.74429&0.946&400\\
\hline
\multicolumn{9}{c}{3 Stoppings, $m=5,10,15$, probit criterion}\\
\hline
0.700&-2&-0.02273&-0.03247&0.00991&~0.56915&0.78539&0.929&305\\
0.700&-1&-0.01782&-0.02546&0.01795&~0.49359&0.87076&0.897&175\\
0.700&~0&-0.00107&-0.00154&0.02818&~0.43507&0.96278&0.845&58\\
0.700&~1&~0.00747&~0.01067&0.03536&~0.41027&1.00466&0.810&15\\
0.700&~2&~0.00779&~0.01113&0.03871&~0.39544&1.02014&0.799&6\\

\hline
\multicolumn{9}{c}{3 Stoppings, $m=5,10,15$, K criterion}\\
\hline
0.900& &0.00038&0.00042&0.00022&0.87113&0.92963&0.950&400\\
\hline
\multicolumn{9}{c}{3 Stoppings, $m=5,10,15$, probit criterion}\\
\hline
0.900&-2&-0.00509&-0.00566&0.00193&~0.84798&0.94184&0.931&359\\
0.900&-1&-0.00977&-0.01086&0.00622&~0.79625&0.98420&0.812&216\\
0.900&~0&-0.00370&-0.00411&0.01172&~0.76136&1.03124&0.591&58\\
0.900&~1&-0.00048&-0.00054&0.01465&~0.75066&1.04837&0.473&10\\
0.900&~2&-0.00050&-0.00056&0.01632&~0.74697&1.05203&0.430&5\\
\hline
\multicolumn{9}{c}{3 Stoppings, $m=2,4,6$, K criterion}\\
\hline
0.001& &0.00004&0.04000&0.00000&-0.00076&0.00284&0.335&400\\
\hline
\multicolumn{9}{c}{3 Stoppings, $m=2,4,6$, probit criterion}\\
\hline
0.001&-2&~0.00006&~0.06167&0.00034&-0.00059&0.00272&0.046&56\\
0.001&-1&~0.00072&~0.72333&0.00062&-0.00091&0.00436&0.046&55\\
0.001&~0&~0.00138&~1.37500&0.00077&-0.00139&0.00614&0.046&54\\
0.001&~1&~0.00138&~1.37500&0.00077&-0.00139&0.00614&0.046&54\\
0.001&~2&~0.00162&~1.62500&0.00095&-0.00141&0.00666&0.046&54\\
\hline
\multicolumn{9}{c}{3 Stoppings, $m=2,4,6$, K criterion}\\
\hline
0.010& &-0.00009&-0.00925&0.00003&~0.00060&0.01922&0.906&400\\
\hline
\multicolumn{9}{c}{3 Stoppings, $m=2,4,6$, probit criterion}\\
\hline
0.010&-2&-0.00500&-0.49992&0.00093&-0.00236&0.01236&0.142&59\\
0.010&-1&-0.00214&-0.21400&0.00238&-0.00343&0.01915&0.142&57\\
0.010&~0&~0.00067&~0.06692&0.00366&-0.00473&0.02607&0.143&54\\
0.010&~1&~0.00272&~0.27175&0.00478&-0.00539&0.03083&0.143&53\\
0.010&~2&~0.00330&~0.33008&0.00518&-0.00547&0.03208&0.143&53\\
\hline
\multicolumn{9}{c}{3 Stoppings, $m=2,4,6$, K criterion}\\
\hline
0.100& &-0.00038&-0.00380&0.00022&0.07037&0.12887&0.950&400\\
\hline
\multicolumn{9}{c}{3 Stoppings, $m=2,4,6$, probit criterion}\\
\hline
0.100&-2&-0.03106&-0.31057&0.01499&-0.00712&0.14500&0.365&90\\
0.100&-1&-0.01085&-0.10845&0.02380&-0.01570&0.19401&0.364&69\\
0.100&~0&~0.00608&~0.06083&0.03219&-0.02219&0.23435&0.362&54\\
0.100&~1&~0.01843&~0.18427&0.04088&-0.01967&0.25652&0.354&47\\
0.100&~2&~0.02697&~0.26972&0.04452&-0.02527&0.27921&0.357&42\\
\hline
\multicolumn{9}{c}{3 Stoppings, $m=2,4,6$, K criterion}\\
\hline
0.300& &~0.00058&~0.00192&0.00054&~0.25571&0.34544&0.946&400\\
\hline
\multicolumn{9}{c}{3 Stoppings, $m=2,4,6$, probit criterion}\\
\hline
0.300&-2&-0.07020&-0.23401&0.03773&~0.07350&0.38610&0.645&154\\
0.300&-1&-0.03432&-0.11439&0.05279&~0.04818&0.48318&0.634&99\\
0.300&~0&~0.00908&~0.03025&0.07342&~0.03309&0.58506&0.601&54\\
0.300&~1&~0.03216&~0.10719&0.08456&~0.02594&0.63837&0.589&31\\
0.300&~2&~0.04043&~0.13477&0.08908&~0.01169&0.66917&0.588&19\\
\hline
\multicolumn{9}{c}{3 Stoppings, $m=2,4,6$, K criterion}\\
\hline
0.500& &0.00089&0.00177&0.00063&0.45195&0.54983&0.943&400\\
\hline
\multicolumn{9}{c}{3 Stoppings, $m=2,4,6$, probit criterion}\\
\hline
0.500&-2&-0.08033&-0.16066&0.04400&~0.25455&0.58479&0.801&233\\
0.500&-1&-0.04041&-0.08082&0.06181&~0.19978&0.71940&0.748&130\\
0.500&~0&~0.00142&~0.00284&0.08358&~0.17317&0.82967&0.684&54\\
0.500&~1&~0.02771&~0.05541&0.09759&~0.17129&0.88412&0.631&18\\
0.500&~2&~0.03284&~0.06567&0.10155&~0.15880&0.90687&0.609&8\\
\hline
\multicolumn{9}{c}{3 Stoppings, $m=2,4,6$, K criterion}\\
\hline
0.700& &-0.00057&-0.00082&0.00054&~0.65456&0.74429&0.946&400\\
\hline
\multicolumn{9}{c}{3 Stoppings, $m=2,4,6$, probit criterion}\\
\hline
0.700&-2&-0.06777&-0.09682&0.03800&~0.49359&0.77086&0.874&289\\
0.700&-1&-0.05260&-0.07515&0.05167&~0.43196&0.86283&0.784&177\\
0.700&~0&-0.00907&-0.01296&0.07342&~0.41494&0.96691&0.601&54\\
0.700&~1&~0.01732&~0.02474&0.08664&~0.43616&0.99848&0.471&11\\
0.700&~2&~0.02245&~0.03207&0.08951&~0.43130&1.01360&0.441&3\\
\hline
\multicolumn{9}{c}{3 Stoppings, $m=2,4,6$, K criterion}\\
\hline
0.900& &0.00038&0.00042&0.00022&0.87113&0.92963&0.950&400\\
\hline
\multicolumn{9}{c}{3 Stoppings, $m=2,4,6$, probit criterion}\\
\hline
0.900&-2&-0.02239&-0.02487&0.01270&~0.81194&0.94329&0.913&350\\
0.900&-1&-0.02509&-0.02788&0.02241&~0.76524&0.98457&0.703&216\\
0.900&~0&-0.00608&-0.00676&0.03219&~0.76565&1.02219&0.362&54\\
0.900&~1&~0.00709&~0.00787&0.03638&~0.78855&1.02562&0.217&8\\
0.900&~2&~0.00467&~0.00519&0.03988&~0.78210&1.02724&0.192&2\\
\hline
\hline
\end{longtable}
\end{center}

\end{document}